\documentclass[a4paper,12pt]{amsart}
\usepackage{amssymb, amsmath,amscd,amsthm}
\usepackage{textcomp}
\usepackage{pifont}
\usepackage{amsfonts}
\usepackage{latexsym}
\usepackage{eufrak}
\usepackage{bm}
\usepackage{enumitem}
\usepackage[utf8x]{inputenc}
\usepackage{verbatim}
\usepackage{upgreek}
\usepackage{mathtools}
\usepackage[all]{xy}
\usepackage{multicol}
\usepackage{graphicx}
\usepackage{stmaryrd}
\usepackage{comment}
\usepackage{xcolor}

\theoremstyle{definition}

\newtheorem*{ack}{Acknowledgements}

\newcommand{\R}{\mathbb{R}}
\newcommand{\Z}{\mathbb{Z}}
\DeclareMathOperator{\A}{\text{Aff}}
\DeclareMathOperator{\I}{\text{Iso}}
\DeclareMathOperator{\G}{\text{GL}}
\DeclareMathOperator{\ort}{\text{O}}
\newcommand{\Id}{\text{Id}}

\theoremstyle{theorem}
\newtheorem{theorem}{Theorem}[section]
\newtheorem{lemma}[theorem]{Lemma}
\newtheorem{proposition}[theorem]{Proposition}
\newtheorem{corollary}[theorem]{Corollary}

\theoremstyle{definition}
\newtheorem{definition}[theorem]{Definition}
\newtheorem{remark}[theorem]{Remark}

\author[Karla Garc\'{\i}a]{Karla Garc\'{\i}a*}
\title{Moduli Spaces of Flat Riemannian Metrics on 3- and 4-dimensional Closed Manifolds}

\begin{document}


\maketitle

\markright{ \uppercase{Moduli Spaces of Flat Riemannian Metrics on Closed Manifolds}}

\begin{abstract}
We describe the topology of the moduli spaces of flat metrics for all the 3-dimensional closed manifolds. We give an algebraic description of the moduli spaces for the 4-dimensional closed flat manifolds with a single generator in their holonomy and, in some cases, also study their topology.     
\end{abstract}

\let\thefootnote\relax\footnotetext{ 


*Current affiliation: \address{Faculty of Science, Universidad Nacional Aut\'onoma de M\'exico (UNAM),}

\hspace{3.1cm} Coyoac\'an, Mexico city, Mexico. 
 
\hspace{0.1cm} E-mail address: \email{ohmu@ciencias.unam.mx} }

\section{Introduction and main results}

A flat manifold is a Riemannian manifold which admits a metric of zero sectional curvature, called a {\bf flat metric}. These manifolds are basically described by their fundamental group, which turns out to be a Bieberbach group \cite{wolf}. 

Let $M$ be a closed manifold. We denote by $\mathcal{R}(M)$ the space of all (complete) Riemannian metrics on $M$. We equip $\mathcal{R}(M)$ with the smooth compact-open topology (\cite{moduli}, \cite{hirsch}, \cite{corroj-b}). We would like to identify the metrics that are isometric, which leads us to consider the following action. Let Diff$(M)$ denote the group of self-diffeomorphisms of $M$, then Diff$(M)$ acts on $\mathcal{R}(M)$ by pulling back metrics. The quotient of $\mathcal{R}(M)$ by this action is the {\bf moduli space $\mathcal{M}(M)$} of Riemannian metrics on $M$. Our main object of study is the following. 

\begin{definition} \label{def moduli}
The {\bf moduli space of flat metrics $\mathcal{M}_{flat}(M)$} is the quotient of the space of (complete) Riemannian metrics with zero sectional curvature, $\mathcal{R}_{flat}(M)$, by the action of Diff$(M)$. 
\end{definition}
 
In \cite{wolf2}, Wolf described $\mathcal{M}_{flat}(M)$ in terms of the Bieberbach group of the manifold. The classification of the Bieberbach groups for dimensions 2 and 3 was also given by Wolf in \cite{wolf}, while for dimension 4 we use the classification given by Lambert in \cite{lambert}.  

For the 3-dimensional closed manifolds, a first study of $\mathcal{M}_{flat}(M)$ was undertaken by Kang in \cite{kang}. Here we do some amendments to Kang's work and complete it by studying the topology of the moduli spaces of flat metrics. Our main result in this case reads as follows.

\begin{theorem} \label{top3}
All moduli spaces of flat metrics over 3-dimensional closed manifolds are contractible, except for two cases, homeomorphic to $\mathbb{S}^1 \times \R^3$.
\end{theorem}

The affine equivalent classes of the 4-dimensional closed flat manifolds are listed by Lambert in \cite{lambert}. As a first step toward the study of the 4-dimensional case, we decided to compute $\mathcal{M}_{flat}(M)$ for the family given by those 4-dimensional closed flat manifolds with one generator in their holonomy.
 
Let us introduce the following notation, which we will use in the Theorem below. Remember that $\G(n, \Z)$ are the integer matrices with determinant $\pm 1$, and consider a matrix of the form 
\[
X= \begin{psmallmatrix}
a & b \\ c & d \end{psmallmatrix} \quad  \text{or} \quad  \begin{psmallmatrix}
a & b & c  \\ 
d & e & f  \\
g & h & i  \end{psmallmatrix} \quad \text{depending on the dimension}.
\]
\begin{equation*}
\begin{aligned}
\Gamma_0 (2) & = \{ X  \in \G(2, \Z) \mid c \equiv 0 \; \text{mod} \; 2\},\\
\Gamma_0 (2)^t & = \{ X^t  \mid X \in \Gamma_0 (2)\},\\  
\Gamma (2) & = \{ X \in \G(2, \Z) \mid c,b \equiv 0 \; \text{and} \; a,d \equiv 1 \; \text{mod} \; 2\}, \\
\Gamma_0(3) & = \left\{ X \in \G(2, \Z) \mid b \equiv 0 \; \text{mod} \; 3 \right\}, \\
\Gamma(3) & = \left\{ X \in \G(2, \Z) \mid b,c \equiv 0 \; \text{mod} \; 3 \right\}, \\
\Gamma_0(4) & =\left\{ X \in \G(2, \Z) \mid b \equiv 0 \; \text{mod} \; 4 \right\}, \\
\Gamma_{0,1}(6)   & =  \left\{ X \in \G(2, \Z) \mid b \equiv 0 \; \text{and} \; d \equiv 1  \; \text{mod} \; 6 \right\},  \\
\Gamma_{0,5}(6)  & =  \left\{ X \in \G(2, \Z) \mid b \equiv 0 \; \text{and} \; d \equiv 5  \; \text{mod} \; 6 \right\}, \\
\Gamma_0(2)_3 & = \left\{ X \in \G(3, \Z) \mid d,g \equiv 0 \; \text{mod} \; 2 \right\}, \\
\Gamma(2)_3 & = \left\{ X \in \G(3, \Z) \mid d,c,f \equiv 0 \; \text{mod} \; 2 \right\}.
\end{aligned}
\end{equation*}

We are using the notation for the flat manifolds given in \cite{lambert}, later in preliminaries we will list them explicitly. 



\begin{theorem} \label{mod4}
The moduli space of flat metrics of the 4-dimensional closed manifolds with a single generator in their holonomy are:
\renewcommand{\labelenumi}{\arabic{enumi}.}
\begin{enumerate}
\item For $T^4$, $\mathcal{M}_{flat} = \ort(4) \backslash \G(4, \R)/ \G(4, \Z)$.
\item For $O^4_2$, $\mathcal{M}_{flat} = (\text{O}(2)\backslash \text{GL}(2,\mathbb{R}) /\text{GL}(2,\mathbb{Z})) \times (\text{O}(2)\backslash \text{GL}(2,\mathbb{R}) /\Gamma_0(2)).$
\item For $O^4_3$, $\mathcal{M}_{flat} = \left( \ort(2) \backslash \G(2, \R) / \Gamma_0(2) \right) \times \left( \ort(2) \backslash \G(2, \R) / \Gamma_0(2)^t \right) $.
\item For $O^4_4$, $\mathcal{M}_{flat} = (\ort(2) \backslash \G(2, \R) / \Gamma_0(3)) \times \R^+$.
\item For $O^4_5$, $\mathcal{M}_{flat}= (\ort(2) \backslash \G(2, \R) / \Gamma(3)) \times \R^+$.
\item For $O^4_6$, $\mathcal{M}_{flat} = \left( \ort(2) \backslash \G(2, \R) / \Gamma_0(4) \right) \times \R^+.$
\item For $O^4_7$, $\mathcal{M}_{flat} = \left( \ort(2) \backslash \G(2, \R) / \Gamma(2) \right) \times \R^+.$
\item For $O^4_8$, $\mathcal{M}_{flat} = ( \ort(2) \backslash \G(2, \R) / \left\langle \Gamma_{0,1}(6), \Gamma_{0,5}(6) \right\rangle ) \times \R^+.$
\item For $N^4_1$, $\mathcal{M}_{flat} = (\text{O}(3)\backslash \text{GL}(3,\mathbb{R}) /\Gamma_0(2)_3 )\times \R^+.$  
\item For $N^4_2$, $\mathcal{M}_{flat}= (\text{O}(3)\backslash \text{GL}(3,\mathbb{R}) /\Gamma(2)_3 ) \times \R^+.$ 
\item For $N^4_{14}$, $\mathcal{M}_{flat} = (\text{O}(3)\backslash \text{GL}(3,\mathbb{R}) /\G(3, \Z)) \times \R^+$.
\item For $N^4_{15}$, $N^4_{16}$, $N^4_{17}$, $N^4_{18}$, $N^4_{19}$, $N^4_{20}$, and $N^4_{21}$,  $\mathcal{M}_{flat} = (\R^+)^3 $.
\end{enumerate} 

\end{theorem}

For some of these cases we may say something about their topology. 

\begin{corollary} \label{top4}
The moduli spaces of flat metrics of the 4-dimensional manifolds with Bieberbach groups $O^4_2$ and $O^4_7$ are non-contractible, specifically they are homeomorphic to $\mathbb{S}^1 \times \R^5$ and to the product of $\R^2$ with a 3-punctured sphere, respectively. On the other hand, the moduli spaces of flat metrics of the 4-dimensional manifolds with Bieberbach groups $N^4_{14}$, $N^4_{15}$, $N^4_{16}$, $N^4_{17}$, $N^4_{18}$, $N^4_{19}$, $N^4_{20}$, and $N^4_{21}$ are contractible.
\end{corollary}


\begin{remark}
The moduli space of flat metrics of the 4-dimensional torus is non-contractible, as proved by Tuschmann and Wiemeler in \cite{tuschwiem}.
\end{remark}

The organization of this paper is as follows. We start in Section 2 with some preliminaries, in Section 3 we explain the descriptions of the moduli spaces of flat metrics for dimension 3 and prove theorem \ref{mod4}, and finally in Section 4 we study their topology for some of the cases. 


\begin{ack} The results in this paper are part of my Ph.D. thesis \cite{Karlsthesis} developed under the supervision of Wilderich Tuschmann. I thank Prof. Tuschmann for presenting me this interesting research line and for his guidance. I thank Oscar Palmas and Ingrid Membrillo for comments on the first versions of the present manuscript, and  for useful conversations. This project was supported by the DFG, Research Training Group 2229.
\end{ack}

\section{Preliminaries}\label{S: Preliminaries}

Here we fix some notation. The group of affine transformations of $\R^n$, denoted by $\A(n)$, has the structure of a semidirect product $\A(n) = \G(n, \R) \ltimes \R^n $. The group of isometries of $\R^n$ denoted by $\I(n)$, also have the structure of a semidirect product $\I(n)= \ort(n)\ltimes \mathbb{R}^n $. 

We work with the following type of groups:
\begin{definition}
A {\bf Bieberbach group $\mathbf{\pi}$} is a discrete subgroup of $\I(n)$ that is torsion-free and such that $\R^n/\pi$ is compact. 
\end{definition}

We have that $(\R^n/\pi, \sigma)$, with $\sigma$ the metric induced from the usual metric of $\mathbb{R}^n$, is a closed flat manifold. On the other hand, let $M$ be a closed manifold with a flat metric $g$, then its universal cover with the metric induced from $g$ is isometric to $\mathbb{R}^n$ with the usual metric. In other words, $\R^n$ with the usual metric is a Riemannian covering of $(M,g)$ and we consider its group of deck transformations, denoted by $\pi$. Then $(M,g)$ is isometric to $(\R^n/\pi, \sigma)$, where $\pi$ is a Bieberbach group. Therefore a closed flat manifold is represented by its Bieberbach group and the Bieberbach theorems describe important properties about them. One of these properties is that two closed flat manifolds with isomorphic fundamental groups, are affinely equivalent. See \cite{charlap} or \cite{wolf}.

Consider the projection homomorphism
$$ 
\begin{array}{cccc}
\tau: & \text{Aff}(n) & \rightarrow & \text{GL}(n, \R) \\ 
      &   (A,v)       & \mapsto     &    A .  
\end{array}      $$

\begin{definition}
Let $\pi$ be a Bieberbach group. The {\bf holonomy} of $\pi$ is the subgroup of $\G(n, \R)$ given by $H_{\pi}\coloneqq  \tau (\pi)$. 
\end{definition}

The  kernel of $\tau$ restricted to $\pi$ is denoted by $L_{\pi}$. It is the maximal normal abelian subgroup of $\pi$, which consists of all the translations (Id, $v)$ of $\pi$. We have a short exact sequence     
\begin{equation} \label{exact seq}
 1 \rightarrow L_{\pi} \rightarrow \pi \rightarrow H_{\pi} \rightarrow 1. 
\end{equation}

We fix some notation in order to give the classification of the Bieberbach groups in the dimensions consider here. We shall denote by $e_1=(1,0,0)$, $e_2=(0,1,0)$ and $e_3=(0,0,1)$ the vectors of the standard basis of $\mathbb{R}^3$. The basic translations of $\mathbb{R}^3$ are denoted by $t_i= (\text{Id}, e_i) $. Also, the rotation matrix by an angle $\theta \in [0, 2\pi]$ is denoted as
$$ R(\theta)= \left( \begin{array}{cc}
\cos (\theta) & -\sin (\theta)    \\ 
\sin (\theta) & \cos (\theta)      \end{array} \right), \; \text{and the reflection as} \; E_0 = \left( \begin{array}{cc}
1 & 0    \\ 
0 & -1      \end{array} \right) .$$

We now enumerate the Bieberbach groups, along with their holonomy and their generators. 

\begin{theorem}[\protect{\cite{wolf}, \cite{kang}}] \label{clasif3}  There are only 10 Bieberbach groups in dimension 3 up to affine change of coordinates. The first six of them give orientable manifolds and the last four give non-orientable manifolds.
\renewcommand{\labelenumi}{\arabic{enumi}.}
\begin{enumerate}
\item $\mathbf{G_1=\text{T}^{\;3}}$: $H_{\pi}=\{ \text{Id} \}$, $\pi= \langle t_1, t_2, t_3 \rangle = \Z^3$.
\item $\mathbf{G_2}$: $H_{\pi}=\Z_2$, $\pi = \langle t_1, t_2, t_3, \alpha=(A, \frac{1}{2}e_1) \rangle$, 
  where $A=   \begin{psmallmatrix}
1 & 0   \\ 
0 & R(\uppi)     \end{psmallmatrix} .$   
\item $\mathbf{G_3}$: $H_{\pi}= \Z_3$, $\pi= \langle t_1, s_1=( \Id,  A(e_2)), s_2=( \Id, A^2(e_2)), \alpha =(A, \frac{1}{3}e_1) \rangle$,  where  $A =  \begin{psmallmatrix}
1 & 0   \\ 
0 & R(\frac{2\uppi}{3})   \end{psmallmatrix} $.
\item $\mathbf{G_4}$: $H_{\pi}= \Z_4$, $\pi= \langle t_1, t_2, t_3, \alpha =(A, \frac{1}{4}e_1 ) \rangle$, 
 where $A=  \begin{psmallmatrix}
1 & 0   \\ 
0 & R(\frac{\uppi}{2})     \end{psmallmatrix} $.
\item $\mathbf{G_5}$: $H_{\pi}=\Z_6$,  $\pi=\langle t_1, s_1=(\Id, A(e_2)), s_2=(\Id, A^2(e_2)), \alpha =(A, \frac{1}{6}e_1) \rangle$,  where $A= \begin{psmallmatrix}
1 & 0   \\ 
0 & R(\frac{\uppi}{3})   \end{psmallmatrix} .$ 
\item $\mathbf{G_6}$: $H_{\pi}=\Z_2 \times \Z_2$, $\pi = \langle t_1, t_2, t_3, \alpha =( A, \frac{1}{2}e_1) , \beta =(B, \frac{1}{2}(e_2 + e_3)) \rangle$,  where   
$A= \begin{psmallmatrix}
1 & 0   \\ 
0 & R(\uppi)     \end{psmallmatrix}$,    
$B= \begin{psmallmatrix}
-1 & 0    \\ 
0  & E_0   \end{psmallmatrix} $. 
\item $\mathbf{B_1}=\mathbb{S}^1 \times K^2$: $H_{\pi}=\Z_2$, $\pi = \langle t_1, t_2, t_3, \epsilon =(E, \frac{1}{2}e_1) \rangle$, where                           
$ E= \begin{psmallmatrix}
1 & 0    \\ 
0 & E_0  \end{psmallmatrix}$.
\item $\mathbf{B_2}$: $H_{\pi}=\Z_2$, $\pi = \langle t_1, t_2, s=(\text{Id}, \frac{1}{2}(e_1 + e_2) + e_3), \epsilon =(E, \frac{1}{2}e_1) \rangle$,  where 
$E$ is the same as in 7.  
 \item $\mathbf{B_3}$: $H_{\pi}=\Z_2 \times \Z_2$, $\pi = \langle t_1, t_2, t_3, \alpha =(A, \frac{1}{2}e_1), \epsilon =(E, \frac{1}{2}e_2) \rangle$,  where \\
$A= \begin{psmallmatrix}
1 & 0   \\ 
0 & R(\uppi)     \end{psmallmatrix} $, and $E$ is the same as in 7.
\item $\mathbf{B_4}$: $H_{\pi}=\Z_2 \times \Z_2$,  $\pi = \langle t_1, t_2, t_3, \alpha =(A, \frac{1}{2}e_1), \epsilon =(E, \frac{1}{2}(e_2 + e_3)) \rangle$,  where \\ 
$ A $ and  $ E$ are the same as in 9. 

\end{enumerate}

\end{theorem} 

We will use analogous notations for dimension 4.

\begin{theorem}[\protect{\cite{lambert}}] \label{clasif4} 
There are 18 Bieberbach groups in dimension 4 up to affine change of coordinates, which have only one generator in their holonomy. The first eight of them give orientable manifolds and the last ten give non-orientable manifolds.
\renewcommand{\labelenumi}{\arabic{enumi}.}
\begin{enumerate}
\item $\mathbf{O^4_1} = T^4$: $H_{\pi}= \{ \text{Id} \}$, $\pi = \langle t_1, t_2, t_3, t_4 \rangle = \mathbb{Z}^4$. 
\item $\mathbf{O^4_2}$: $H_{\pi}=\mathbb{Z}_2$,
$\pi = \langle t_1, t_2, t_3, t_4, \alpha = (A, \frac{1}{2}e_4) \rangle $,
 where $A= \begin{psmallmatrix}
R(\uppi) & 0  \\
 0 & \text{Id}
  \end{psmallmatrix}. $
\item $\mathbf{O^4_3}$: $H_{\pi}=\mathbb{Z}_2$,
$\pi = \langle t_1, t_2, t_3, s= (\text{Id} , \frac{1}{2}(e_1 + e_4)) , \alpha = (A, \frac{1}{2}e_2) \rangle $, where \\ 
$A= \begin{psmallmatrix}
\text{Id} & 0 \\ 
0 & R(\uppi)   
\end{psmallmatrix}. $
\item $\mathbf{O^4_4}$: $H_{\pi}=\mathbb{Z}_3$, 
$\pi = \langle t_1, t_2, t_3, s= (\text{Id} , \frac{1}{2}e_3 + \frac{\sqrt{3}}{2} e_4)) , \alpha = (A, \frac{1}{3}e_2) \rangle $, 
 where  $A=  \begin{psmallmatrix}
\text{Id} & 0 \\ 
0 & R(\frac{2\uppi}{3})   
\end{psmallmatrix}. $
\item$\mathbf{O^4_5}$: $H_{\pi}=\mathbb{Z}_3$, 
$\pi = \langle t_1, t_2, s_1=(\text{Id} , -\frac{1}{3}e_2 + \frac{2\sqrt{3}}{3} e_3), s_2= (\text{Id} , \frac{1}{3}e_2 + \frac{\sqrt{3}}{3} e_3 + e_4), \\ \alpha = (A, \frac{1}{3}e_1) \rangle $, where $A$ is the same as in 4.
\item $\mathbf{O^4_6}$: $H_{\pi}=\mathbb{Z}_4$, 
$\pi = \langle t_1, t_2, t_3, t_4 , \alpha = (A, \frac{1}{4}e_2) \rangle $,
 where $A= \begin{psmallmatrix}
\text{Id} & 0 \\ 
0 & R(\frac{\uppi}{2})   
\end{psmallmatrix}. $
\item $\mathbf{O^4_7}$: $H_{\pi}=\mathbb{Z}_4$, 
$\pi = \langle t_1, t_2, s_1=(\text{Id} , \frac{1}{2}e_1 + \frac{1}{2}e_2 +e_3), s_2= (\text{Id} , \frac{1}{2}e_1 + \frac{1}{2} e_2 + e_4) , \\
\alpha = (A, \frac{1}{4}e_2) \rangle $,
 where $A$ is the same as in 6.
\item $\mathbf{O^4_8}$: $H_{\pi}=\mathbb{Z}_6$, 
$\pi = \langle t_1, t_2, t_3, s= (\text{Id} , \frac{1}{2}e_3 + \frac{\sqrt{3}}{2} e_4) , \alpha = (A, \frac{1}{6}e_2) \rangle $, where \\ 
$A= \begin{psmallmatrix}
\text{Id} & 0 \\ 
0 & R(\frac{\uppi}{3})   
\end{psmallmatrix}. $
\item $\mathbf{N^4_1}=K^2 \times T^2$: $H_{\pi}=\mathbb{Z}_2$, 
$\pi = \langle t_1, t_2, t_3, t_4 , \alpha = (A, \frac{1}{2}e_1) \rangle $,
 where $A= \begin{psmallmatrix}
\text{Id} & 0 \\ 
0 & E_0   
\end{psmallmatrix}. $
\item $\mathbf{N^4_2}$: $H_{\pi}=\mathbb{Z}_2$, 
$\pi = \langle t_1, t_2, t_3, s=( \text{Id}, \frac{1}{2}(e_3 + e_4)) , \alpha = (A, \frac{1}{2}e_1) \rangle $, where  $A $ is the same as in 9.
\item $\mathbf{N^4_{14}}$: $H_{\pi}=\mathbb{Z}_2$, 
$\pi = \langle t_1, t_2, t_3, t_4 , \alpha = (A, \frac{1}{2}e_4) \rangle $,
 where $A= \begin{psmallmatrix}
-\text{Id} & 0 \\ 
0 & -E_0   
\end{psmallmatrix}. $
\item $\mathbf{N^4_{15}}$: $H_{\pi}=\mathbb{Z}_4$, 
$\pi = \langle t_1, t_2, t_3, t_4 , \alpha = (A, \frac{1}{4}e_2) \rangle $,
 where $A= \begin{psmallmatrix}
-E_0 & 0 \\ 
0 & R(\frac{\uppi}{2})   
\end{psmallmatrix}. $
\item $\mathbf{N^4_{16}}$: $H_{\pi}=\mathbb{Z}_4$, 
$  \pi =  \langle s_1=(\text{Id}, {\scriptstyle \frac{1}{2}}(e_1+ e_2 +e_3)), s_2=(\text{Id}, {\scriptstyle \frac{1}{2}}(-e_1+ e_2 -e_3)) , \\
s_3=(\text{Id}, {\scriptstyle \frac{1}{2}}(-e_1- e_2 +e_3)), t_4 , \alpha = (A, {\scriptstyle \frac{1}{4}}e_4) \rangle , $ where $A= \begin{psmallmatrix}
- R(\frac{\uppi}{2}) & 0 \\
     0   &   -E_0     
\end{psmallmatrix}. $
\item $\mathbf{N^4_{17}}$: $H_{\pi}=\mathbb{Z}_4$, 
$\pi = \langle t_1, t_2, t_3, t_4 , \alpha = (A, \frac{1}{2}e_2) \rangle $,
 where $A= \begin{psmallmatrix}
0 & 1 & 0 \\
1 & 0 & 0 \\ 
0 & 0 & R(\frac{\uppi}{2})   
\end{psmallmatrix}. $
\item $\mathbf{N^4_{18}}$: $H_{\pi}=\mathbb{Z}_4$, 
$\pi = \langle t_1, t_2, s_1=(\text{Id}, {\scriptstyle \frac{1}{2}}(e_1+ e_2)+e_3), s_2=(\text{Id}, {\scriptstyle \frac{1}{2}}(e_1+ e_2)+e_4) , \alpha = (A, \frac{1}{4}(e_1 + e_2) + \frac{1}{2}e_3) \rangle $,
 where $A= \begin{psmallmatrix}
E_0 & 0 \\ 
0 & R(\frac{3\uppi}{2})   
\end{psmallmatrix}. $
\item $\mathbf{N^4_{19}}$: $H_{\pi}=\mathbb{Z}_6$, 
$\pi = \langle t_1, t_2, t_3 , s=(\text{Id}, \frac{1}{2}e_3 + \frac{\sqrt{3}}{2}e_4), \alpha = (A, \frac{1}{6}e_2) \rangle $, 
 where  $A= \begin{psmallmatrix}
-E_0 & 0 \\ 
0 & R(\frac{2\uppi}{3})   
\end{psmallmatrix}. $
\item $\mathbf{N^4_{20}}$: $H_{\pi}=\mathbb{Z}_6$, 
$\pi = \langle t_1, t_2, t_3 , s=(\text{Id}, \frac{1}{2}e_3 + \frac{\sqrt{3}}{2}e_4), \alpha = (A, \frac{1}{6}e_2) \rangle $,
 where  $A= \begin{psmallmatrix}
-E_0 & 0 \\ 
0 & -R(\frac{2\uppi}{3})   
\end{psmallmatrix}. $
\item $\mathbf{N^4_{21}}$: $H_{\pi}=\mathbb{Z}_6$, 
$\pi = \langle t_1, t_2, t_3 , t_4 , \alpha = (A, \frac{1}{6}e_1) \rangle $,
 where  $A= \begin{psmallmatrix}
1 & 0 & 0  & 0 \\ 
0 & 0 & -1  & 0 \\
0 & 0 & 0 &  -1  \\
0 & -1 & 0 & 0
  \end{psmallmatrix}. $
\end{enumerate}
\end{theorem}

For some of the cases we conjugate the representation in order to get a matrix in $\G(4, \Z)$ as the generator of the holonomy. With an abuse of notation, we denote them as before in the following lemma.  
\begin{lemma} \label{integerrep}
For some of the Bieberbach groups in Theorem \ref{clasif4}, we conjugate them in order to get the following representations.
\begin{itemize}
\item[$\triangleright$] For $O^4_4$, $\pi = \langle t_1, t_2, t_3, t_4 , \alpha = (A, \frac{1}{3}e_2) \rangle $, where $A= \begin{psmallmatrix}
\text{Id} & 0  & 0 \\ 
0 & 0 & -1  \\
0 & 1 & -1  \end{psmallmatrix}. $
\item[$\triangleright$] For $O^4_5$, $\pi = \langle t_1, t_2, s_1=(\text{Id} ,{\scriptstyle -\frac{1}{3}}e_2 -{\scriptstyle \frac{2\sqrt{3}}{3}} e_3 - {\scriptstyle\frac{2\sqrt{3}}{3}} e_4), s_2= (\text{Id} ,{\scriptstyle \frac{1}{3}}e_2 -{\scriptstyle \frac{2}{\sqrt{3}}} e_4 ) , \\
\alpha = (A, {\scriptstyle \frac{1}{3}}e_1) \rangle , $  where  $A= \begin{psmallmatrix}
\text{Id} & 0  & 0 \\ 
0 & 0 & -1  \\
0 & 1 & -1   \end{psmallmatrix}. $
\item[$\triangleright$] For $O^4_8$, $\pi = \langle t_1, t_2, t_3, t_4 , \alpha = (A, \frac{1}{6}e_2) \rangle $, where $A= \begin{psmallmatrix}
\text{Id} & 0  & 0 \\ 
0 & 0 & -1  \\
0 & 1 & 1  \end{psmallmatrix}. $
\item[$\triangleright$] For $N^4_{16}$, $\pi = \langle t_1, t_2, t_3, t_4 , \alpha = (A, \frac{1}{4}e_4) \rangle $, where $A= \begin{psmallmatrix}
-1 & 1 & 0  & 0 \\ 
-1 & 0 & 1  & 0 \\
-1 & 0 & 0 &  0  \\
0 & 0 & 0 & 1
  \end{psmallmatrix}. $
\item[$\triangleright$] For $N^4_{18}$, $\pi = \langle t_1, t_2, t_3, t_4 , \alpha = (A, \frac{1}{2}e_3) \rangle $, where $A= \begin{psmallmatrix}
1 & 0 & 1  & 0 \\ 
0 & -1 & 0  & -1 \\
0 & 0 & 0 &  1  \\
0 & 0 & -1 & 0
  \end{psmallmatrix}. $
\item[$\triangleright$] For $N^4_{19}$, $\pi = \langle t_1, t_2, t_3, t_4 , \alpha = (A, \frac{1}{6}e_2) \rangle $, where $A= \begin{psmallmatrix}
-E_0 & 0  & 0 \\ 
0 & -1 &  -1  \\
0 & 1 & 0   \end{psmallmatrix}. $
\item[$\triangleright$] For $N^4_{20}$, $\pi = \langle t_1, t_2, t_3, t_4 , \alpha = (A, \frac{1}{6}e_2) \rangle $, where $A= \begin{psmallmatrix}
-E_0 & 0  & 0 \\ 
0 & 1 &  1  \\
0 & -1 & 0  \end{psmallmatrix}. $
\end{itemize} 
\end{lemma}
\begin{proof}
We change the representation by conjugating with the affine transformation
$(P,0)$, where $P$ is: 
\[
\text{For} \; O^4_4 \; \text{and} \; O^4_5, \; P = \begin{psmallmatrix}
\text{Id} & 0  & 0 \\ 
0 & -1 & \frac{1}{\sqrt{3}}  \\
0 & -1 & -\frac{1}{\sqrt{3}} 
  \end{psmallmatrix}. \quad \text{For} \; O^4_8,\; N^4_{19} \; \text{and} \; N^4_{20}, \; P= \begin{psmallmatrix}
\text{Id} & 0  & 0 \\ 
0 & 1 & -\frac{1}{\sqrt{3}}  \\
0 & 0 & \frac{2}{\sqrt{3}}
  \end{psmallmatrix}.  
\]
\[
\text{For} \; N^4_{16}, \; P= \begin{psmallmatrix}
0 & 1 & 1  & 0 \\ 
-1 & 1 & 0  & 0 \\
-1 & 0 & 1 &  0  \\
0 & 0 & 0 & 1
  \end{psmallmatrix}. \quad \text{For} \; N^4_{18}, \; P = \begin{psmallmatrix}
1& 0 & -\frac{1}{2}  & -\frac{1}{2} \\ 
0 & 1 & -\frac{1}{2} & -\frac{1}{2}  \\
0 & 0 & 1 & 0  \\
0 & 0 & 0 & 1 
  \end{psmallmatrix}. 
\]
\end{proof}

The next result gives a description of the moduli space of flat metrics on a manifold depending only on its Bieberbach group $\pi$. We denote the normalizer of $\pi$ in $\A(n)$ by   
$\text{N}_{\text{Aff}(n)} (\pi) =\{ \gamma \in \text{Aff}(n) \mid \gamma \pi \gamma^{-1} = \pi \}.$  

\begin{theorem}[\protect{\cite{wolf2}}{ \bf Wolf}] \label{Wolf}
The subset 
$$\text{Iso}(n) \backslash \{ \gamma \in \text{Aff}(n) \mid \gamma \pi \gamma^{-1} \subset  \text{Iso} (n) \} /  \text{N}_{\text{Aff}(n)} (\pi)$$
of the double coset space Iso$(n) \backslash $ Aff$(n) / $ N$_{\text{Aff}(n)} (\pi)$, is in bijective correspondence with the set of all isometry classes of Riemannian manifolds that are affinely equivalent to $M=\R^n/ \pi$. The double coset 
Iso$(n)\cdot\gamma\cdot$N$_{\text{Aff}(n)} (\pi)$ corresponds to the isometry class of $\R^n/(\gamma \pi \gamma^{-1})$.    
\end{theorem}

Actually, there is an homeomorphism between this double quotient and the moduli space given in definition \ref{def moduli}; details can be found in \cite{Karlsthesis}.
 
The translation part does not bring any additional information to the expression of $\mathcal{M}_{flat}(M)$. Therefore, the moduli space of flat metrics of $M=\R^n/\pi$ is 
$$   \ort(n) \backslash \{ X \in \text{GL}(n, \R) \mid XH_{\pi}X^{-1} \subset \text{O}(n)  \} / \tau ( \text{N}_{\text{Aff}(n)} (\pi)). $$

\subsection{Notations} \label{notations}
In the next sections we use the following notations:
\begin{itemize}
\item[$\triangleright$] $H_1 \cdot H_2 :=\{ h_1\cdot h_2 \mid h_1\in H_1, \; h_2 \in H_2 \}$, where $H_1$ and $H_2$ are two subgroups of a given group.
\item[$\triangleright$] The cone space $C_{\pi}= \{ X \in \text{GL}(n, \R) \mid XH_{\pi}X^{-1} \subset \text{O}(n)  \}$.  
\item[$\triangleright$] The matrix part of the normalizer $\mathcal{N}_{\pi}= \tau ( \text{N}_{\text{Aff}(n)} (\pi))$. 
\item[$\triangleright$] The lattice of $\pi$, $L_{\pi}$, consists of all the translations (Id, $v)$ of $\pi$. The standard lattice is  $\Z^n=\langle (\Id , e_i) \mid \{ e_1, \dots , e_n \} \; \text{the standard basis of} \; \R^n \rangle $.
\end{itemize}

\section{Algebraic description}

In this section we give information of the moduli spaces $\mathcal{M}_{flat}(M)$ for the 3-dimensional closed manifolds and compute them for the family of 4-dimensional closed manifolds. First we need the representation of the Bieberbach group given in theorems \ref{clasif3} and \ref{clasif4}. Then we use Theorem \ref{Wolf} in order to describe the moduli space of flat metrics, which we express in the previous section as 
$$  \mathcal{M}_{flat}(\R^n/\pi)= \ort(n) \backslash C_{\pi} / \mathcal{N}_{\pi}. $$
Let us analyse the structure of the cone space and the matrix part of the normalizer.

\subsection{The cone space} 
The cone space $C_{\pi}$ is easy to analyze since it only depends on the holonomy. To describe the space $C_{\pi}$, one has to solve the equation: 
\begin{equation} \label{eqcone}
X \in \text{GL}(n, \R) \quad \text{such that} \quad (X^{t}X)A = A(X^{t}X) 
\end{equation}
for all $A\in H_{\pi}$ (see \protect{\cite[Lemma 2.2]{kangkimteich}}).

For the descriptions of $C_{\pi}$ in dimension 3 we refer to \cite{kangkimor} and \cite{kangkimteich}. We will give the description  for the 4-dimensional  closed flat manifolds with one generator in their holonomy. 

\begin{proposition} \label{conespacedim4}
The possible spaces $C_{\pi}$ for the 4-dimensional closed flat manifolds with a single generator in their holonomy are the following: 
\renewcommand{\labelenumi}{\arabic{enumi}.}
\begin{enumerate}
\item For trivial holonomy: $T^4$, the space is $C_{\pi}= \G(4, \R)$.
\item For $H_{\pi}= \Z_2$, the spaces are:  
\begin{enumerate}
\item For $O^4_2$ and $O^4_3$, $C_{\pi}= \ort(4) \cdot ( \G(2, \R)\times \G(2, \R) )$.
\item For $N^4_1$, $N^4_2$, and $N^4_{14}$,  $C_{\pi}=  \ort(4) \cdot ( \G(3, \R) \times \R^*)$.
\end{enumerate}

\item For cyclic holonomy of order bigger than 2, the spaces are:
\begin{enumerate}
\item For $O^4_4$, $O^4_5$, $O^4_6$, $O^4_7$ and $O^4_8$, $C_{\pi} = \ort(4) \cdot ( \G(2, \R) \times (\R^+ \times \ort(2)))$.
\item For $N^4_{15}$, $N^4_{18}$, $N^4_{19}$ and $N^4_{20}$, $C_{\pi} =  \ort(4) \cdot ((\R^+)^2 \times \ort(2)) \times (\R^+ \times \ort(2)))$. \\
For $N^4_{16}$, $C_{\pi} =  \ort(4) \cdot ((\R^+ \times \ort(2) \times (\R^+)^2 \times \ort(2)))$. \\
For $N^4_{17}$, $C_{\pi} =  \ort(4) \cdot ((\R^+ \times (0, \pi) \times \ort(2) ) \times (\R^+ \times \ort(2)))$. 
\item For $N^4_{21}$, $C_{\pi} =  \ort(4) \cdot (\R^* \times (\R^+ \times (0, \frac{2\pi}{3}) \times \ort(3)))$.
\end{enumerate}
\end{enumerate}
\end{proposition}

\begin{proof}
We consider each case separately.

\noindent Case 1. When the holonomy is trivial, we have the result in the Corollary of Theorem 1 in \cite{wolf2}.

\noindent Case 2. For $H_{\pi}= \Z_2$. \\
When $H_{\pi}$ is generated by $A= \begin{psmallmatrix}
- \text{Id} & 0 \\ 
 0 & \text{Id}
  \end{psmallmatrix} $ or its negative $-A$, which is the case of $O^4_2$ and $O^4_3$, we get from equation \eqref{eqcone} that 
\begin{equation*}
\begin{aligned}
C_{\pi} & = \left\{ (x_1 \; x_2 \; x_3 \; x_4) \in \G(4, \R) \mid x_i \perp x_j \; \text{for} \; i=1,2 \; \text{and} \; j=3,4 \right\} \\
    & = \ort(4) \cdot \left( \begin{array}{cc} \G(2, \R) & 0 \\
0 & \G(2, \R) 
\end{array} \right). 
\end{aligned}
\end{equation*}   
When the generator of $H_{\pi}$ is $A= \begin{psmallmatrix}
\text{Id} &  0  \\
 0 & -1
  \end{psmallmatrix}$ or its negative $-A$, which is the case of $N^4_1$, $N^4_2$, and $N^4_{14}$, we have 
\begin{equation*}
\begin{split}
C_{\pi} & = \left\{ (x_1 \; x_2 \; x_3 \; x_4) \in \G(4, \R) \mid x_4 \perp x_i \; \text{with} \; i=1,2,3 \right\} \\
   & =  \ort(4) \cdot \left\{ \left( \begin{array}{cc}
B & 0 \\
0 & a
\end{array} \right) \mid a\in \R^* \; \text{and} \; B\in \G(3, \R) \right\} \\
  &  = \ort(4) \cdot ( \G(3, \R) \times \R^*).
\end{split}
\end{equation*}

\noindent Case 3. For cyclic holonomy with order bigger than 2 we use the property that if $A\in H_{\pi}$ and $X\in \G(n, \R)$ such that $ XAX^{-1} \in \ort(n)$ then $ XA^rX^{-1}=(XAX^{-1})^r \in \ort(n)$ for any $r\in \mathbb{N}$. \\
When $H_{\pi}$ is generated by matrices of the form $\begin{psmallmatrix}
\text{Id} & 0  \\ 
 0 & R(\theta)     \end{psmallmatrix}$, which is the case of $O^4_4$, $O^4_5$, $O^4_6$, $O^4_7$, and $O^4_8$, we get from equation \eqref{eqcone} that 
\begin{equation*}
\begin{split}
C_{\pi} & = \left\{ (x_1 \; x_2 \; x_3 \; x_4) \in \G(4, \R) \mid x_i \perp x_j \; \text{for} \; i=1,2 \; \text{and} \; j=3,4, \; \text{with} \; x_3 \perp x_4 \right. \\
  & \left. \; \text{and} \; \| x_3 \| = \| x_4 \| \right\} \\
      & = \ort(4) \cdot \left( \begin{array}{cc} \G(2, \R) & 0 \\
0 & \R^+ \times \ort(2) 
\end{array} \right) \\
         & = \ort(4) \cdot ( \G(2, \R) \times (\R^+ \times \ort(2))).
\end{split}
\end{equation*}
When the holonomy is generated by matrices of the form $\begin{psmallmatrix}
-E_0 & 0  \\ 
0 & R(\theta)     \end{psmallmatrix}$ or  $\begin{psmallmatrix}
R(\theta) & 0  \\ 
0 & -E_0       \end{psmallmatrix}$, which is the case of $N^4_{15}$, $N^4_{16}$, $N^4_{18}$, $N^4_{19}$, and $N^4_{20}$, we have 
\begin{equation*}
\begin{split}
C_{\pi} & = \left\{ (x_1 \; x_2 \; x_3 \; x_4) \in \G(4, \R) \mid x_i \perp x_j \; \text{for all} \; i \neq j \; \text{with} \; i,j \in \{1,2,3,4\}, \right. \\
  & \left. \; \text{and} \; \| x_3 \| = \| x_4 \| \right\} \\
   &  =  \ort(4) \cdot \left( ((\R^+)^2 \times \ort(2)) \times (\R^+ \times \ort(2)) \right).
\end{split}
\end{equation*}
When the holonomy is generated by $\begin{psmallmatrix}
0 & 1 & 0  \\
1 & 0 & 0  \\ 
0 & 0 & R(\theta)     \end{psmallmatrix}$, as in $N^4_{17}$, we have
\begin{equation*}
\begin{split}
C_{\pi} & = \left\{ (x_1 \; x_2 \; x_3 \; x_4) \in \G(4, \R) \mid x_i \perp x_j \; \text{for} \; i=1,2 \; \text{and} \; j=3,4, \; \text{with} \; \| x_1 \| = \| x_2 \|  \; \right. \\
  & \left. \; x_3 \perp x_4, \; \text{and} \; \| x_3 \| = \| x_4 \|  \right\}.
\end{split}
\end{equation*}
Then the vectors $x_1$ and $x_2$ have the same length and the angle between them should be smaller than $\pi$. Thus $C_{\pi} =  \ort(4) \cdot ((\R^+ \times (0, \pi) \times \ort(2) ) \times (\R^+ \times \ort(2))). $ \\  
When the holonomy is generated by  $\begin{psmallmatrix}
1 & 0 & 0  & 0 \\ 
0 & 0 & -1  & 0 \\
0 & 0 & 0 &  -1  \\
0 & -1 & 0 & 0
  \end{psmallmatrix}$, which is the case of $N^4_{21}$, we have
\begin{equation*}
\begin{split}
C_{\pi} & =  \left\{ (x_1 \; x_2 \; x_3 \; x_4) \in \G(4, \R) \mid x_1 \perp x_i  \; \text{with} \; i= 2,3,4, \right. \\ 
 &  \left.  \|x_2\|=\|x_3\|=\|x_4\| \; \text{and} \; x_2\cdot x_3= x_2 \cdot x_4 =x_3 \cdot x_4 \right\} . \\
\end{split}
\end{equation*}  
This means that the vectors $x_2$, $x_3$ and $x_4$ have the same length and form the same angle between them. For this situation we have that the angle is $\theta \in (0, \frac{2\pi}{3})$ since having angle $\frac{2\pi}{3}$ means that the vectors are coplanar (and not linearly independent anymore). With this information we can conclude that
\[
C_{\pi} =  \ort(4) \cdot (\R^* \times (\R^+ \times (0, {\scriptstyle \frac{2\pi}{3}}) \times \ort(3)).
\]

\end{proof}

\subsection{The matrix part of the normalizer}
The description of the normalizer depends not only on the holonomy but also on the affine structure as well, i.e., on how the translations are acting. To get easier computations for some cases we change the representation by conjugating with a suitable affine transformation. These cases are: $G_3$, $G_5$, where the representation is changed as in \cite[Lemma 2.2]{kang}; $O^4_4$, $O^4_5$, $O^4_8$, $N^4_{16}$, $N^4_{18}$, $N^4_{19}$ and $N^4_{20}$, where the representation is changed as in Lemma \ref{integerrep}. We observe that if $\pi'=\xi \pi \xi^{-1}$ where $\xi \in \A(n)$ and $\pi$ a Bieberbach group, we have that the normalizer behaves as follows $\text{N}_{\A(n)} ( \pi' ) = \xi \text{N}_{\A(n)} (\pi) \xi^{-1}$ (\protect{\cite[page 1069]{kang}). 

By \eqref{exact seq}, we always have a lattice $L_{\pi}$ inside our Bieberbach group $\pi$, and $\tau (\text{N}_{\A(n)} (L_{\pi}))$ is $\G(n, \Z)$ or a conjugate of $\G(n, \Z)$ by the matrix of change of coordinates when the lattice $L_{\pi}$ is not $\Z^n$. To keep our notation simple, we will assume that the lattice is the standard one for the next explanation. Then we have $\mathcal{N}_{\pi} \subseteq  $N$_{\G(n, \Z)}(H_{\pi})$. We may have the following two situations: $\mathcal{N}_{\pi}$ is not always N$_{\G(n, \Z)}(H_{\pi})$ and the normalizer is not always a semidirect product. Having the following property of $\pi$ will make $\mathcal{N}_{\pi}$ easier to describe.

\begin{definition}
Let $\pi$ be a Bieberbach group with non trivial holonomy. We say that the group has {\bf translation part not involved}, when for $X\in \text{N}_{\G(n,\Z)} (H_{\pi})$ we have that $X(v)=v=(v_1, \dots, v_n)$ or $X(v)=(u_1, \dots, u_n)$, with $u_i= -v_i$ for some $i \in I \subseteq \{1, \dots, n\}$ and $u_j=v_j$ for $j \notin I$; for each generator $\alpha=(A,v)$ of $\pi$ such that $A \neq \Id$. Otherwise, we say it has {\bf translation part  involved}.
\end{definition}

In the Bieberbach groups we are studying, having translation part not involved, standard lattice and for each generator $\alpha=(A,v)$ we have $(A, -v) \in \pi$, then $\mathcal{N}_{\pi} =  \text{N}_{\G(n, \Z)}(H_{\pi}).$ When we do not have the properties mentioned before, which is most of the cases, we can still see if the normalizer has a structure of semidirect product using the following lemma.   




\begin{lemma} \label{semipro}
Let $G$ be a subgroup of $\A(n)$. For all $(X,u) \in G$, we have $(X,0) \in G$ if and only if $G= M \ltimes T$, where $M$ is the matrix part and $T$ are the translations of $G$.   
\end{lemma}

Since $G$ has the product from $\A(n)$, the previous lemma is actually telling us when $G$ can be split into a product $M\times T$. Then, for proving the lemma one can use the splitting theorem.    


In general, we have to look for matrices in  $\text{N}_{\G(n,\Z)} (H_{\pi})$ which preserve the translations of any generator $\alpha=(A,v) \in \pi$ with $A \neq \Id$, i.e., all the possible options for a vector $u\in \R^n$ such that  $(A,u) \in \pi$.  


We proceed with the description of $\mathcal{N}_{\pi}$ for the 3-dimensional Bieberbach groups.

\begin{proposition}[\protect{\cite{kang}}] \label{normalizerdim3}
Let $\pi$ be one of the Bieberbach groups for the 3-dimensional closed flat manifolds, then the matrix part of the normalizer of $\pi$, $\mathcal{N}_{\pi}$, in $\A(3)$ are as follows:
\vspace{-0.5cm}
\begin{multicols}{2}
\renewcommand{\labelenumi}{\arabic{enumi}.}
\begin{enumerate}
\item For $T^3$, $\mathcal{N}_{\pi}= \G(3, \Z)$.
\item For $G_2$, \\
$\mathcal{N}_{\pi}= \left\{ \begin{psmallmatrix}
\pm 1 & 0  \\ 
0 & B     \end{psmallmatrix} \mid B\in \G(2, \Z) \right\}$. 
\item For $G_3$, \\
$\mathcal{N}_{\pi}= D_6 = \left\langle \begin{psmallmatrix}
1 & 0   \\ 
0 & R(\frac{\uppi}{3})   \end{psmallmatrix}, \begin{psmallmatrix}
-1 & 0    \\ 
0  & E_0   \end{psmallmatrix} \right\rangle $.
\item For $G_4$, \\
$\mathcal{N}_{\pi}= D_4 = \left\langle  \begin{psmallmatrix}
1 & 0   \\ 
0 & R(\frac{\uppi}{2})     \end{psmallmatrix}, \begin{psmallmatrix}
-1 & 0    \\ 
0  & E_0   \end{psmallmatrix} \right\rangle $.
\item For $G_5$, \\
$\mathcal{N}_{\pi}= D_6 =  \left\langle \begin{psmallmatrix}
1 & 0   \\ 
0 & R(\frac{\uppi}{3})   \end{psmallmatrix}, \begin{psmallmatrix}
R(\uppi) & 0 \\
 0       & 1     \end{psmallmatrix} \right\rangle $.  

\columnbreak

\item For $G_6$, $\mathcal{N}_{\pi}= \left\{ \begin{psmallmatrix}
\pm 1 & 0 & 0 \\ 
0 & \pm 1 & 0 \\    
0 &  0 & \pm 1 \\   \end{psmallmatrix} \rtimes \mathcal{S}_3  \right\},$
where $\mathcal{S}_3$ is the  permutation group of 3 letters.
\item For $B_1$,  $\mathcal{N}_{\pi}= \left\{ \begin{psmallmatrix}
\Gamma_{0}(2) & 0  \\
 0 & \pm 1 \\    \end{psmallmatrix} \right\}$.
\item For $B_2$, \\
$\mathcal{N}_{\pi}= \left\{ \begin{psmallmatrix}
\Gamma(2) & 0  \\
 0 & \pm 1 \\    \end{psmallmatrix} \cdot \left\langle \begin{psmallmatrix}
0 & 1 & 0  \\
1 & 0 & 0  \\
 0 & 0 & -1 \\    \end{psmallmatrix} \right\rangle  \right\}$.
\item For $B_3$ and $B_4$, \\
$\mathcal{N}_{\pi}= \left\{ \begin{psmallmatrix}
\pm 1 & 0 & 0 \\ 
0 & \pm 1 & 0 \\    
0 &  0 & \pm 1 \\   \end{psmallmatrix} \right\}$.   
\end{enumerate}
\end{multicols}
\end{proposition}


Although the above result was proved in \cite[Lemma 3.3]{kang}, we point out and correct a mistake in the cited reference while calculating $\mathcal{N}_{\pi}$ for the group $B_1$. The group $B_1$ has standard lattice and the normalizer has structure of semidirect product. Now, we have to be careful with the translation part of the generator $\epsilon$; this means we have to restrict to matrices in N$_{\G(3, \Z)}(H_{\pi})$ that preserve the corresponding lattice of the generator $\epsilon$:
$$ \epsilon L_{\pi} = \{ (E, v) \mid v= \frac{2n_1 + 1}{2} e_1 + n_2e_2  + n_3 e_3, \; \text{with} \; n_1, n_2, n_3 \in \Z\} . $$ 
Then we look for $X \in \G( 2, \Z)$ such that $X(\frac{2n_1 +1}{2}, n_2)^t=(\frac{2k_1 +1}{2}, k_2)$, with $n_i$, $k_i$ $\in \Z$ for $i=1,2$. This only happens for matrices in $\Gamma_0(2)$, getting the conclusion. 

We continue computing the matrix part of the normalizer for the 4-dimensional  closed flat manifolds with one generator in their holonomy. We analyze separately the orientable and the non-orientable manifolds. 

Let us introduce the following notation, which we will use in the coming two propositions.
Again we consider 
\[
X= \begin{psmallmatrix}
a & b \\ c & d \end{psmallmatrix} \quad  \text{or} \quad  \begin{psmallmatrix}
a & b & c  \\ 
d & e & f  \\
g & h & i  \end{psmallmatrix} \quad \text{depending on the dimension}.
\]

$\begin{aligned} 
\Gamma_{0,1}(3)& = \left\{ X \in \G(2, \Z) \mid b \equiv 0 \; \text{and} \; d \equiv 1  \; \text{mod} \; 3 \right\} , \\
\Gamma_{0,2}(3) & = \left\{ X \in \G(2, \Z) \mid b \equiv 0 \; \text{and} \; d \equiv 2  \; \text{mod} \; 3 \right\} , \\
\Gamma_{1,2}(3)  & = \left\{ X \in \G(2, \Z) \mid b,c \equiv 0, \;  a\equiv 1 \; \text{and} \; d \equiv 2  \; \text{mod} \; 3 \right\}, \\
\Gamma_{2,1}(3) & = \left\{ X \in \G(2, \Z) \mid b,c \equiv 0, \;  a\equiv 2 \; \text{and} \; d \equiv 1  \; \text{mod} \; 3 \right\}, \\
\Gamma_{0,1}(4) & = \left\{ X \in \G(2, \Z) \mid b \equiv 0 \; \text{and} \; d \equiv 1  \; \text{mod} \; 4 \right\} , \\
\Gamma_{0,3}(4) & = \left\{ X \in \G(2, \Z) \mid b \equiv 0 \; \text{and} \; d \equiv 3  \; \text{mod} \; 4 \right\} , \\
\Gamma_{0,1}(2,4)  & = \left\{ X \in \G(2, \Z) \mid c \equiv 0 \; \text{mod} \; 2, \; b \equiv 0 \; \text{and} \; d \equiv 1  \; \text{mod} \; 4 \right\} , \\
\Gamma_{0,3}(2,4) & =  \left\{ X \in \G(2, \Z) \mid  c \equiv 0 \; \text{mod}\; 2, \; b \equiv 0 \; \text{and} \; d \equiv 3  \; \text{mod} \; 4 \right\} , \\
\widehat{\Gamma}(2)_3 & = \left\{ X \in \G(3, \Z) \mid d,g,c,f \equiv 0 \; \text{mod} \; 2 \right\}.
\end{aligned}$

\begin{proposition} \label{normalizadordim4or}
The matrix part of the normalizer of $\pi$ in $\A(4)$ for the 4-dimensional orientable closed flat manifolds with a single generator in their holonomy is as follows:
\renewcommand{\labelenumi}{\arabic{enumi}.}
\begin{enumerate}
\item For $T^4$, $\mathcal{N}_{\pi}= \G(4, \Z)$.
\item For $O^4_2$, $\mathcal{N}_{\pi} = \left\{ \begin{psmallmatrix} \G(2, \Z) & 0 \\
0 & \Gamma_0(2)
\end{psmallmatrix} \right\}$.
\item For $O^4_3$, $\mathcal{N}_{\pi} = \left\{ \begin{psmallmatrix}
 \Gamma(2) & 0   \\
0 & \Gamma_0(2)^t
  \end{psmallmatrix} \cdot \left\langle  \begin{psmallmatrix}
1 & 1 & 0   \\ 
0 & 1 & 0   \\
0 & 0 & \text{Id}  
  \end{psmallmatrix} \right\rangle \right\}$.
\item For $O^4_4$, $\mathcal{N}_{\pi}= \left\langle 
\begin{psmallmatrix} B & 0  \\
0 &  R(\frac{\uppi}{3})   \\
\end{psmallmatrix}
, \begin{psmallmatrix} C & 0  \\
0 & E_0
\end{psmallmatrix} \mid  B \in \Gamma_{0,1}(3), C\in \Gamma_{0,2}(3) \right\rangle $.
\item For $O^4_5$, $\mathcal{N}_{\pi}= \left\langle 
\begin{psmallmatrix} B & 0  \\
0 &  R(\frac{\uppi}{3})   \\
\end{psmallmatrix}
, \begin{psmallmatrix} C & 0  \\
0 & E_0
\end{psmallmatrix} \mid  B \in \Gamma_{1,2}(3), C\in \Gamma_{2,1}(3) \right\rangle $.
\item For $O^4_6$, $ \mathcal{N}_{\pi}= \left\langle \begin{psmallmatrix}
B & 0   \\
0 & R(\frac{3\uppi}{2})
  \end{psmallmatrix}, \begin{psmallmatrix}
C & 0  & 0 \\
0 & 0 & 1  \\
0 & 1 & 0
  \end{psmallmatrix}  \mid B \in  \Gamma_{0,1}(4), C\in  \Gamma_{0,3}(4)   \right\rangle$.
\item For $O^4_7$, \\
$\mathcal{N}_{\pi} =\left\{ \left\langle \begin{psmallmatrix}
B & 0   \\
0 & R(\frac{3\uppi}{2})
  \end{psmallmatrix}, \begin{psmallmatrix}
C & 0  & 0 \\
0 & 0 & 1  \\
0 & 1 & 0
  \end{psmallmatrix} \mid B \in  \Gamma_{0,1}(2,4), C \in  \Gamma_{0,3}(2,4) \right\rangle \right\} \cdot \left\langle \begin{psmallmatrix}
1 & 2 & 0  \\
0 & 1 & 0  \\
0 & 0 & E_0
\end{psmallmatrix} \right\rangle $.
\item For $O^4_8$, $\mathcal{N}_{\pi} = \left\langle \begin{psmallmatrix}
B & 0   \\
0 & R(\frac{\uppi}{3}) \\
  \end{psmallmatrix}, \begin{psmallmatrix}
C & 0   \\
0 & -E_0
  \end{psmallmatrix} \mid B \in \Gamma_{0,1}(6), C \in \Gamma_{0,5}(6)  \right\rangle  $.
\end{enumerate}
\end{proposition}

\begin{proof}
In the case of T$^4$, the result follows from Corollary of Theorem 1 in \cite{wolf2}. Therefore we exclude the torus from our analysis. 

The lattices of the Bieberbach groups $O^4_3$, $O^4_5$, and $O^4_7$ are not the standard ones. Then the group of matrices that normalizes the lattice is conjugate to $\G(4, \Z)$ by a matrix $Q \in \G(4, \R)$. For these cases the $Q$ is computed, but fortunately  the $X\in Q\G(4, \Z)Q^{-1}$ that satisfy the condition $XA=AX$ for the generator of the holonomy $A$ are reduced to matrices in $\G(4, \Z)$. Then in all cases we can consider matrices in $\G(4, \Z)$.

We first find all the matrices $X \in \G(4, \Z)$ that normalize the holonomy $H_{\pi}$. For all cases we get that the matrix must have the form  
$$X = \left( \begin{array}{cc} X_1 & 0 \\
0 & X_2 
\end{array} \right), \quad \text{with} \quad X_1, \; X_2 \in \G(2, \Z).$$

It turns out that the translation part is involved for all the cases. Then the lattices of the generators of the holonomy have to be computed and we have to search for matrices $X\in \text{N}_{\G(4, \Z)}(H_{\pi})$ that preserve or switch the lattices.

For $O^4_2$, the matrix $X_2 \in \G(2, \Z)$ must preserve  vectors of the form \\ 
$X_2(n_3, \frac{2n_4+1}{2})^t= (k_3, \frac{2k_4+1}{2}) $, with $n_i$, $k_i \in \Z$ for $i=3,4$, similar to the case of $B_1$.

For cyclic holonomy of order $k$ bigger than 2, we have the cases
$$ XAX^{-1}= \left\{ \begin{array}{cc}
A &    \\
A^r & \text{with} \; r\in \mathbb{N}, 1 < r < k \; \text{and} \; (r, k)=1 .
\end{array} \right. 
$$
For the ones with standard lattice, $O^4_4$, $O^4_6$ and $O^4_8$, we look for the matrices $X_1$  such that $X_1(n_1, \frac{kn_2 + 1}{k})= (k_1, \frac{kk_2 + 1}{k})$ or $X_1(n_1, \frac{kn_2 + 1}{k})= (k_1, \frac{kk_2 + r}{k})$, with $n_i$, $k_i \in \Z$ for $i=1,2$, depending on if we are fixing the generator $A$ or switching it to the generator $A^r$.  
The matrices $X_2$ are the same as in the cases of dimension 3 with their respective holonomy.     

We have more cases for the ones with non-standard lattice. Let us see this more closely:

For $O^4_5$ the lattices of the generators are:

\smallskip
$\begin{aligned}
\alpha L_{\pi} = \{ (A, v) \mid v= \frac{3n_1 +1}{3} e_1 + \frac{3n_2 - n_3 + n_4}{3}e_2 +(n_3+ n_4){\scriptstyle \frac{2}{\sqrt{3}}}e_3  + n_4{\scriptstyle \frac{2}{\sqrt{3}}} e_4 \; \\
  \text{and} \; n_i \in \Z, i=1,2,3,4 \}  
\end{aligned}$

$\begin{aligned} 
\alpha^2 L_{\pi} =  \{ (A^2, v) \mid v= \frac{3n_1 +2}{3} e_1 + \frac{3n_2 - n_3 + n_4}{3}e_2 -n_4{\scriptstyle \frac{2}{\sqrt{3}}}e_3  + n_3{\scriptstyle \frac{2}{\sqrt{3}}} e_4 \; \text{and} \\
\; n_i \in \Z, i=1,2,3,4 \}.  
\end{aligned}$\\
\noindent We have the next three options:\\
1. $-n_3 + n_4 \in 3\Z$, \hspace{0.5cm} 2. $-n_3 + n_4 \in 3\Z+1$,\hspace{0.5cm} 3. $-n_3 + n_4 \in 3\Z+2$.\\
\noindent Looking at all combinations for sending the lattices, it is concluded that not all of them are possible, leading us to get the structure of semidirect product in the normalizer.

$O^4_3$ and $O^4_7$ are the only ones whose normalizer do not accept a structure of semidirect product. We consider each case separately.

For $O^4_3$, the lattice of the generator is:  
\begin{equation*} 
\begin{split}
\alpha L_{\pi} = \{ (A, v) \mid v= \frac{2n_1 + n_4}{2} e_1 + \frac{2n_2 + 1}{2} e_2 - n_3e_3  -\frac{n_4}{2} e_4 \; \text{and} \; n_i \in \Z, \\ i=1,2,3,4 \}.  
\end{split}
\end{equation*}
We have two cases:  $n_4$ odd or $n_4$ even. Looking at all possibilities for the translations of the generators, we obtain:
$$ \text{N}_{\A(4)}(O^4_3) = \left( \left\{ \left( \begin{array}{cc}
 \Gamma(2) & 0   \\
0 & \Gamma_0(2)^t
  \end{array} \right) \right\} \ltimes \R \oplus \R \oplus {\scriptstyle \frac{1}{2}} \Z \oplus {\scriptstyle \frac{1}{2}} \Z \right) \cdot \langle \xi \rangle ,$$
  
\noindent where $ \xi = \left( \left( \begin{array}{cccc}
1 & 1 & 0  & 0 \\ 
0 & 1 & 0  & 0 \\
0 & 0 & 1 & 0  \\
0 & 0 & 0  & 1
  \end{array} \right) , \frac{1}{4} e_4 \right)  $.\\ 

\bigskip

For $O^4_7$, the lattices of the generators are as follows:
\begin{equation*}
\begin{split} 
\alpha L_{\pi} = \{ (A, v) \mid v= {\scriptstyle \frac{2n_1+n_3+n_4}{2}}e_1 +{\scriptstyle \frac{4n_2+2(n_3+n_4)+1}{4}}e_2 - n_4e_3 + n_3 e_4  \;  \text{and} \; n_i \in \Z,\\
 i=1,2,3,4 \}.  
\end{split}
\end{equation*}

\begin{equation*}
\begin{split} 
\alpha^3 L_{\pi} = \{ (A^3, v) \mid v= { \scriptstyle\frac{2n_1 + n_3 +n_4}{2}}e_1 +{\scriptstyle \frac{4n_2 +2n_3+2n_4+3}{4}}e_2 + n_4e_3  - n_3 e_4 \;  \text{and} \; n_i \in \Z, \\
 i=1,2,3,4 \}.  
\end{split}
\end{equation*}
We will have two cases: $ n_3 +n_4 $ even or  $ n_3 +n_4$ odd. Looking at all possibilities for the translations of the generators, we obtain:\\

\noindent $ \text{N}_{\A(4)}(O^4_7) =$

\begin{equation*}
\begin{split}
\left( \left\langle \left( \begin{array}{ccc}
\Gamma_{0,1}(2,4)  & 0  & 0 \\
   0               & 0  & 1  \\
   0               & -1 & 0
  \end{array} \right), \left( \begin{array}{ccc}
\Gamma_{0,3}(2,4)  & 0  & 0 \\
   0               & 0  & 1  \\
   0               & 1 & 0
  \end{array} \right)  \right\rangle  \ltimes \R \oplus \R \oplus T \right) \cdot \langle \xi \rangle ,
\end{split}  
\end{equation*} 

\noindent where \\ 
$\xi = \left( \left( \begin{array}{cccc}
1 & 2 & 0 & 0 \\
0 & 1 & 0 & 0 \\
0 & 0 & 1 & 0 \\
0 & 0 & 0 & -1  
\end{array} \right) , (0,0,\frac{1}{2}, -\frac{1}{2}) \right), \; \text{and} \quad T= \{ (t, n-t) \mid t\in {\scriptstyle \frac{1}{2}} \Z \; 
 \text{and} \; n\in \Z \} .$

\end{proof}

\begin{proposition} \label{normalizerdim4non-or}
The matrix part of the normalizer of $\pi$ in $\A(4)$ for the 4-dimensional non-orientable closed flat manifolds with a single generator in their holonomy is as follows:
\renewcommand{\labelenumi}{\arabic{enumi}.}
\begin{enumerate}
\item For $N^4_1$, $\mathcal{N}_{\pi}= \left\langle \begin{psmallmatrix}
B & 0  \\
0  & \pm 1
  \end{psmallmatrix} \mid B \in \Gamma_0(2)_3 \right\rangle $.
\item For $N^4_2$, $\mathcal{N}_{\pi} =\left\{ \left\langle \begin{psmallmatrix}
B & 0   \\
0 & \pm 1    \end{psmallmatrix} \mid B \in  \widehat{\Gamma}(2)_3  \right\rangle \right\} \cdot \left\langle \begin{psmallmatrix}
1 & 0 & 0 & 0 \\
0 & 1 & 0 & 0 \\
1 & 0 & 1 & 0 \\
0 & 0 & 0 & 1  
\end{psmallmatrix} \right\rangle $.
\item For $N^4_{14}$, $\mathcal{N}_{\pi} = \left\{ \begin{psmallmatrix} \G(3, \Z) & 0 \\
0 & \pm 1
\end{psmallmatrix} \right\}$.
\item For $N^4_{15}$, $\mathcal{N}_{\pi} =  \left\langle \begin{psmallmatrix}
\pm 1 & 0 & 0   \\ 
0 & 1 & 0   \\
0 & 0 & R(\frac{\uppi}{2})
  \end{psmallmatrix}, \begin{psmallmatrix}
\pm 1 & 0  & 0 \\ 
0 & -1 &  0 \\
0 & 0 & E_0
  \end{psmallmatrix} \right\rangle.$
\item For $N^4_{16}$, $\mathcal{N}_{\pi} =$
$  \left\langle \begin{psmallmatrix}
 R(\frac{\uppi}{2})  & 0 & 0 \\ 
 0 & \pm 1 & 0 \\
 0 & 0    &  1  \end{psmallmatrix}, \begin{psmallmatrix}
E_0 & 0 & 0 \\ 
0 & \pm 1 & 0 \\
0 &  0   & -1  \end{psmallmatrix} \right\rangle.$
\item For $N^4_{17}$, $\mathcal{N}_{\pi} =  \left\langle \begin{psmallmatrix}
B & 0    \\ 
0 & R(\frac{\uppi}{2})
  \end{psmallmatrix}, \begin{psmallmatrix}
 B & 0 \\ 
 0 & E_0
  \end{psmallmatrix} \mid B \in \left\{ \pm \text{Id}, \pm \begin{psmallmatrix}
0 & 1  \\ 
1 & 0  \\  \end{psmallmatrix} \right\} \right\rangle.$
\item For $N^4_{18}$, $\mathcal{N}_{\pi} =  \left\langle \begin{psmallmatrix}
\pm 1 & 0 & 0   \\ 
0 & \pm 1 & 0   \\
0 & 0 & R(\frac{\uppi}{2})
  \end{psmallmatrix}, \begin{psmallmatrix}
\pm 1 & 0  & 0 \\ 
0 & \pm 1 &  0 \\
0 & 0 & E_0
  \end{psmallmatrix} \right\rangle.$
\item For $N^4_{19}$, $\mathcal{N}_{\pi} = \left\langle \begin{psmallmatrix}
\pm 1 & 0 & 0  \\ 
0 & 1 & 0  \\
0 & 0 & R(\frac{\pi}{3}) \\
  \end{psmallmatrix}, \begin{psmallmatrix}
\pm 1 & 0 & 0   \\ 
0 & R(\uppi) & 0   \\
0 &     0    & 1
  \end{psmallmatrix} \right\rangle.$
\item For $N^4_{20}$, $\mathcal{N}_{\pi} =  \left\langle \begin{psmallmatrix}
\pm 1 & 0 & 0  \\ 
0 & 1 & 0  \\
0 & 0 & R(\frac{5\pi}{3}) \\
  \end{psmallmatrix}, \begin{psmallmatrix}
\pm 1 & 0 & 0  \\ 
0 & -1 & 0   \\
0 & 0 & -E_0
  \end{psmallmatrix} \right\rangle.$
\item For $N^4_{21}$, $\mathcal{N}_{\pi} =  \left\langle \begin{psmallmatrix}
1 & 0 & 0  & 0 \\ 
0 & 0 & -1  & 0 \\
0 & 0 & 0 & -1  \\
0 & -1 & 0  & 0
  \end{psmallmatrix},  \begin{psmallmatrix}
-1 & 0 & 0  & 0 \\ 
0 & 0 & 1  & 0 \\
0 & 1 & 0 & 0  \\
0 & 0 & 0  & 1
  \end{psmallmatrix} \right\rangle.$ 
 \end{enumerate}
\end{proposition}

\begin{proof}
First, we explain the case of $N^4_1$. The group $N^4_1$ has standard lattice, translation part involved, and its normalizer has structure of semidirect product. The matrix that normalizes the holonomy has to be of the form 
$$X= \left( \begin{array}{cc}
X_1 & 0 \\
0 & \pm 1
\end{array} \right), \quad \text{with} \quad X_1 \in \G(3, \Z).$$ 
For $X$ to be in $\mathcal{N}_{\pi}$, it also has to preserve the lattice of the generator $\alpha$
\begin{equation*}
\begin{split} 
\alpha L_{\pi} = \{ (A, v) \mid v= \frac{2n_1 +1}{2} e_1 + n_2e_2 + n_3e_3  - n_4e_4 \; 
  \text{and} \; n_i \in \Z, \\
  i=1,2,3,4 \}.  
\end{split}
\end{equation*}

The case of $N^4_2$ is similar to $N^4_1$ but its group has non-standard lattice. Then the form of the matrix $X$ is the same but the lattice of the generator $\alpha$ is different:
\begin{equation*}
\begin{split} 
\alpha L_{\pi} = \{ (A, v) \mid v= \frac{2n_1+1}{2} e_1 + n_2e_2 + \frac{2n_3+n_4}{2}e_3 -\frac{n_4}{2}e_4 \; 
  \text{and} \; n_i \in \Z, \\
  i=1,2,3,4 \},  
\end{split}
\end{equation*}

\noindent with two cases: $n_4$ even or $n_4$ odd. This leads us to have elements in the normalizer that needs the translation part different from zero. Therefore the whole normalizer group is
\begin{equation*}
\begin{split}
\text{N}_{\A(4)}(N^4_2) & = \left( \left\langle \left( \begin{array}{cccc}
2a+1 & b & 2c  & 0 \\ 
2d & 2e+1 & 2f  & 0 \\
2g & h & 2i+1 & 0  \\
0 & 0 & 0  & \pm 1
  \end{array} \right) \in \G(4,\Z) \right\rangle  
   \ltimes T \right) \cdot \langle \xi \rangle, 
\end{split}  
\end{equation*}  
with $a,b,c,d,e,f,g,h,i \in \Z$, the translations $T= \R \oplus \R \oplus \R \oplus {\scriptstyle \frac{1}{2}}\Z$, and  
$$\xi = \left( \left( \begin{array}{cccc}
1 & 0 & 0 & 0 \\
0 & 1 & 0 & 0 \\
1 & 0 & 1 & 0 \\
0 & 0 & 0 & 1  
\end{array} \right) , (0,0,0, \frac{1}{4}) \right) .$$ 

\bigskip

The case of $N^4_{14}$ is simple because the form of the matrix $X$ is also as in $N^4_1$ but now the group $N^4_{14}$ has standard lattice, translation part not involved, and for the generator $\alpha$ we have $(A, -\frac{1}{2}e_4) \in \pi$, then $\mathcal{N}_{\pi}$ is the same as N$_{\G(4, \Z)}(H_{\pi})$.

The remaining groups but $N^4_{18}$ are also simple to compute since they have standard lattice, translation part not involved and N$_{\G(4, \Z)}(H_{\pi})$ is finite, which we computed using Mathematica. Then, we just have to select the matrices that send the lattices of the generators correctly.

The case of $N^4_{18}$ has the generator $(A, \frac{1}{4}(e_1 + e_2) + \frac{1}{2}e_3)$, this means that the rotation of the matrix affects the translation. Even though we change the representation to get standard lattice, the translations of the generators are a bit more complicated, that is why we have to check for each $P \in $N$_{\G(4, \Z)}(H_{\pi})$ if there is an $x \in \R^4$ such that $(P,x)$ normalize $\pi$.    
\end{proof}

\subsection{Moduli spaces}

Having the descriptions of $C_{\pi}$ and $\mathcal{N}_{\pi}$, we can describe the moduli spaces of flat metrics, which we need in order to study their topology: 

\begin{theorem}[\protect{\cite{kang}}] \label{moduli3}
The moduli space of flat metrics of the 3-dimensional closed manifolds are the following:
\renewcommand{\labelenumi}{\arabic{enumi}.}
\begin{enumerate}
\item For $T^3$, $\mathcal{M}_{flat}= \ort(3) \backslash \G(3, \R) / \G(3, \Z)$.
\item For $G_2$, $\mathcal{M}_{flat} = \R^{+} \times ( \ort(2) \backslash \G(2, \R) / \G(2,\Z))$. 
\item For $G_3$, $G_4$ and $G_5$, $\mathcal{M}_{flat} = (\R^+)^2$.
\item For $B_1$, $\mathcal{M}_{flat} =  ( \ort(2) \backslash \G(2, \R) / \Gamma_0(2)) \times \R^{+}$.
\item For $B_2$, $\mathcal{M}_{flat} =  ( \ort(2) \backslash \G(2, \R) / \Gamma(2)\cdot \left\langle Y \right\rangle) \times \R^{+}$, where $Y= \begin{psmallmatrix}
0  & 1 \\ 1 & 0 \end{psmallmatrix}$.
\item For $G_6$, $B_3$ and $B_4$, $\mathcal{M}_{flat} = (\R^{+})^3 $. 
\end{enumerate}
\end{theorem}

We proceed to describe the moduli spaces of flat metrics for the 4-dimensional closed flat manifolds with one generator in their holonomy.
\begin{proof}[Proof of Theorem \ref{mod4}]
As we have seen   
$$  \mathcal{M}_{flat}= \ort(4) \backslash C_{\pi} / \mathcal{N}_{\pi} ,$$
and we already have described the spaces $C_{\pi}$ (Proposition \ref{conespacedim4}) and $\mathcal{N}_{\pi}$ (Proposition \ref{normalizadordim4or} and \ref{normalizerdim4non-or}), so we just have to put all the information together.

For the orientable manifolds of cyclic holonomy of order greater than 2 their double quotient has the form
$$  \ort(4) \backslash \ort(4) \cdot (\G(2, \R)\times (\R^+\times \ort(2))) \Big/ \left\langle \begin{psmallmatrix}
\Gamma_1 & 0  \\
0 & R
  \end{psmallmatrix} , \begin{psmallmatrix}
\Gamma_2 & 0  \\
0 & A
  \end{psmallmatrix} \right\rangle , $$

\noindent where $\Gamma_1$, $\Gamma_2 \subset \G(2, \Z)$, and $R$, $A \in \ort(2)$, the respective matrices that appear in $\mathcal{N}_{\pi}$ for each case. Observe that $\begin{psmallmatrix}
C & 0  & 0 \\
0 & 1 & 0  \\
0 & 0 & 1
  \end{psmallmatrix} \notin \mathcal{N}_{\pi} $, where $C \in \Gamma_2$; this means that $\mathcal{N}_{\pi}$ can not be separated as the product of the groups. But we still can factorize the double quotient as follows: 
$$ \left( \ort(2) \backslash \G(2, \R) / \left\langle \Gamma_1, \Gamma_2 \right\rangle \right) \times \left( \ort(2) \backslash \R^+ \times \ort(2) / \left\langle R, A \right\rangle \right) , $$   
this is because the second part of the space $C_{\pi}$ is $\R^+ \times \ort(2)$ and $\left\langle R, A \right\rangle$, the second factor of the group $\mathcal{N}_{\pi}$ is finite and generated by orthogonal matrices. Then we separate the double quotient into two factors and reduce the second factor as in Theorem \ref{moduli3}.

For the non-orientable manifolds with cyclic holonomy of order greater than 2, we can reduce the double quotient because the normalizer is a subgroup of $\ort(4)$ and the cone space $C_{\pi}$ is equal to orthogonal matrices times the positive real numbers. Let us see the case of $N^4_{15}$:

\noindent$\mathcal{M}_{flat}(N^4_{15})  =$

\noindent $\ort(4) \backslash \ort(4) \cdot \left( ((\R^+)^2 \times \ort(2)) \times (\R^+ \times \ort(2)) \right) \Big/ \left\langle \begin{psmallmatrix}
\pm 1 & 0 & 0  & 0 \\ 
0 & 1 & 0  & 0 \\
0 & 0 & 0 & -1  \\
0 & 0 & 1 & 0
  \end{psmallmatrix}, \begin{psmallmatrix}
\pm 1 & 0 & 0  & 0 \\ 
0 & -1 & 0  & 0 \\
0 & 0 & 1 & 0  \\
0 & 0 & 0  & -1
  \end{psmallmatrix} \right\rangle $\\
  
  $  = \R^+ \times \R^+ \times \left( \ort(2) \backslash \R^+ \times \ort (2) / \left\langle \begin{psmallmatrix}
 0 & -1  \\
 1  & 0
  \end{psmallmatrix}, \begin{psmallmatrix}
 1 & 0  \\
  0  & -1
  \end{psmallmatrix} \right\rangle \right) $\\
  
  $\cong \R^+ \times \R^+ \times \R^+.$ 
\end{proof}

\section{Topological description}

In this section we study the topology of the moduli space $\mathcal{M}_{flat}(M)$ for closed manifolds in dimension 3 and some cases in dimension 4. This is related to the study of the action of subgroups of SL$(2, \Z)$ on the hyperbolic plane. 

$\mathcal{M}_{flat}(M)$ can be seen also as a quotient of the Teichm\"uller space by the group $\mathcal{N}_{\pi}$. In \cite{bettiol} Bettiol, Derdzinski and Piccione studied the Teichm\"uller space of flat manifolds proving that it is always a Euclidean space. Since $\mathcal{N}_{L_{\pi}}$ is a conjugate of $\G(n, \mathbb{Z})$ inside $\G(n, \R)$, which is discrete, and $\mathcal{N}_{\pi} \subset \mathcal{N}_{L_\pi}$ then $\mathcal{N}_{\pi}$ is discrete. Thus the Teichm\"uller space and $\mathcal{M}_{flat}(M)$ are orbifolds with the same dimension. Even though $\mathcal{N}_{\pi}$ is a discrete group acting on a Euclidean space, it turns out that $\mathcal{M}_{flat}(M)$ can have interesting topology.

The Teichm\"uller space of flat metrics and $\mathcal{M}_{flat}(M)$ of the 2-torus are very well understood, see \cite{mapclassgr}; in the cited reference it is shown the existence of an homeomorphism 
\begin{equation} \label{homeohyp}
\ort(2) \backslash \G(2, \R) \cong \R^+ \times \mathbb{H}^2 .
\end{equation}
Since $ \mathcal{M}_{flat} ( \text{T}^2)= \text{O}(2)\backslash \text{GL}(2,\mathbb{R}) /\text{GL}(2,\mathbb{Z})$, we still need to see what happens to the action of $\text{GL}(2,\mathbb{Z})$. Observe that we are quotienting out the orientation reversing matrices, therefore we just have to consider the group $\text{SL}(2, \mathbb{Z})$ acting on the previous space. The action of $\text{SL}(2, \mathbb{Z})$ on $\mathbb{H}^2$ is via M\"obius transformations, and we can even compute the fundamental domain to get the next result 
$$ \mathcal{M}_{flat} ( \text{T}^2)\cong \R^+ \times (\mathbb{H}^2 /\text{SL}(2,\mathbb{Z})) \cong \R^+ \times (\mathbb{S}^2 \setminus \{*\}) \cong \R^3. $$

To study the topology of the double quotient of some of our flat manifolds we have to compute the fundamental domain of the action of a subgroup of  SL$(2, \Z)$ on the hyperbolic plane. We use the fact that $\text{SL}(2, \mathbb{Z})$ has two generators $S= \begin{psmallmatrix}
0 & -1 \\ 1 & 0 \end{psmallmatrix}$ and $T= \begin{psmallmatrix}
1 & 1 \\ 0 & 1 \end{psmallmatrix}$, where the map $S$ is an inversion together with a reflection and the map $T$ is just a translation. For general information about this see \cite{apostol}, \cite{milne} or \cite{shimura}. 

The  algorithm to compute the fundamental domain of a subgroup $\Gamma$ of SL$(2, \Z)$ on $\mathbb{H}^2$, deduced from Proposition 2.16  in \cite{milne}, is:

\begin{enumerate}
\item Compute the index of $\Gamma$ in SL$(2, \Z)$.
\item Find representatives in SL$(2, \Z)$ for $\Gamma$.
\item Apply the respective representative transformations to the fundamental domain of SL$(2, \Z)$ on $\mathbb{H}^2$. 
\end{enumerate}

We can proceed with the proof of the remaining results: 

\begin{proof}[Proof of Theorem \ref{top3}]

The proof is done case by case, using the descriptions of the moduli spaces of flat metrics given in Theorem \ref{moduli3}.

\noindent For the 3-torus we have $\mathcal{M}_{flat}(T^3)= \ort(3) \backslash \G(3, \R) / \G(3, \Z).$ This is contractible by the work of Soul\'e \cite{soule}.\\
\noindent For $G_2$ we have 
\begin{equation*}
\begin{aligned}
\mathcal{M}_{flat}(G_2) & = \R^{+} \times ( \ort(2) \backslash \G(2, \R) / \G(2,\Z)) \\
  & \cong \R^+ \times (\R^+ \times \mathbb{H}^2 / \text{SL}(2, \Z)) \\
  & \cong (\R^+)^2 \times \mathbb{S}^2 \setminus \{*\},
\end{aligned}
\end{equation*}
as we saw for the 2-torus.

The next cases are clearly contractible:

\noindent For $G_k$, with $k=3,4,5$, we have $\mathcal{M}_{flat}(G_k) = (\R^+)^2 .$
 
\noindent For $G_6$, $B_i$, with $i=3,4$, we have $\mathcal{M}_{flat}(G_6)= \mathcal{M}_{flat}(B_i) = (\R^{+})^3 . $

  The non-contractible cases are the following:

\noindent For $B_1$, we have 
\begin{equation*}
\begin{aligned}
\mathcal{M}_{flat}(B_1) & =  ( \ort(2) \backslash \G(2, \R) / \Gamma_0(2)) \times \R^{+} \\
     &  \cong (\R^+ \times \mathbb{H}^2 / \Gamma_0(2)^+ ) \times \R^+, \\
\end{aligned}
\end{equation*}
where $\Gamma_0(2)^+$ are the matrices in $\Gamma_0(2)$ with positive determinant. We compute the fundamental domain of $\Gamma_0(2)^+$ on $\mathbb{H}^2$:

1. We will use that the index of $\Gamma(2)^+$ in SL$(2, \Z)$, $[ \text{SL}(2, \Z) : \Gamma(2)^+]$, is 6, see \cite{shimura}, pages 20-22.

\noindent We have that $\Gamma(2)^+ < \Gamma_0(2)^+ < \text{SL}(2, \Z)$, then 
$$ [ \text{SL}(2, \Z) : \Gamma(2)^+]= [ \text{SL}(2, \Z) : \Gamma_0(2)^+][\Gamma_0(2)^+ : \Gamma(2)^+],$$
since the index is multiplicative. This means that $[ \text{SL}(2, \Z) : \Gamma_0(2)^+] \leq 3$, because $[\Gamma_0(2)^+ : \Gamma(2)^+] \neq 1$. 

\noindent On the other hand, we have that for any group $G$ and subgroup $H<G$, if $g\in G$, then $g^n \in H$ with $n=[G:H]$. \\
Consider $B= \left( \begin{array}{cc}
2 & 1 \\ 1 & 1 \end{array} \right) \notin \Gamma_0(2)^+$. Then $B^2 \notin \Gamma_0(2)^+$ but $B^3 \in  \Gamma_0(2)^+$. This means that  $[ \text{SL}(2, \Z) : \Gamma_0(2)^+] \geq 3$. Therefore,
$[ \text{SL}(2, \Z) : \Gamma_0(2)^+] = 3$.

2. Consider\\
$\gamma_1= \Id$, \hspace{0.5cm} $\gamma_2=  \left( \begin{array}{cc}
0 & -1 \\ 1 & 0 \end{array} \right)= S$, \hspace{0.5cm} $\gamma_3= \left( \begin{array}{cc}
0 & -1 \\ 1 & 1 \end{array} \right)=ST$.\\

\noindent They are representatives of $\Gamma_0(2)^+$ since:
\begin{equation*}
\begin{aligned}
\Gamma_0 (2)^+ & = \left\{ \left( \begin{array}{cc}
2a +1 & b \\ 2c & 2d +1 \end{array} \right)  \in \G(2, \Z) \mid a,c,d \in \Z \right\} ,\\
\Gamma_0 (2)^+ \gamma_2 & = \left\{ \left( \begin{array}{cc}
a  & 2b+1 \\ 2c+1 & 2d \end{array} \right)  \in \G(2, \Z) \mid b,c,d \in \Z \right\} ,\\
\Gamma_0 (2)^+ \gamma_3 & = \left\{ \left( \begin{array}{cc}
a  & b \\ 2c+1 & 2d +1 \end{array} \right)  \in \G(2, \Z) \mid c,d \in \Z \right\} .\\
\end{aligned}
\end{equation*}

3. It is enough to see what the two transformations $S$ and $T$ are doing to the fundamental domain of SL$(2, \Z)$ on $\mathbb{H}^2$. Then we make the corresponding compositions to obtain the fundamental domain shown in Figure  \ref{fund0}.


\hspace{3cm}
\begin{figure}[h]
\begin{center}
\scalebox{0.3}[0.3]{\includegraphics{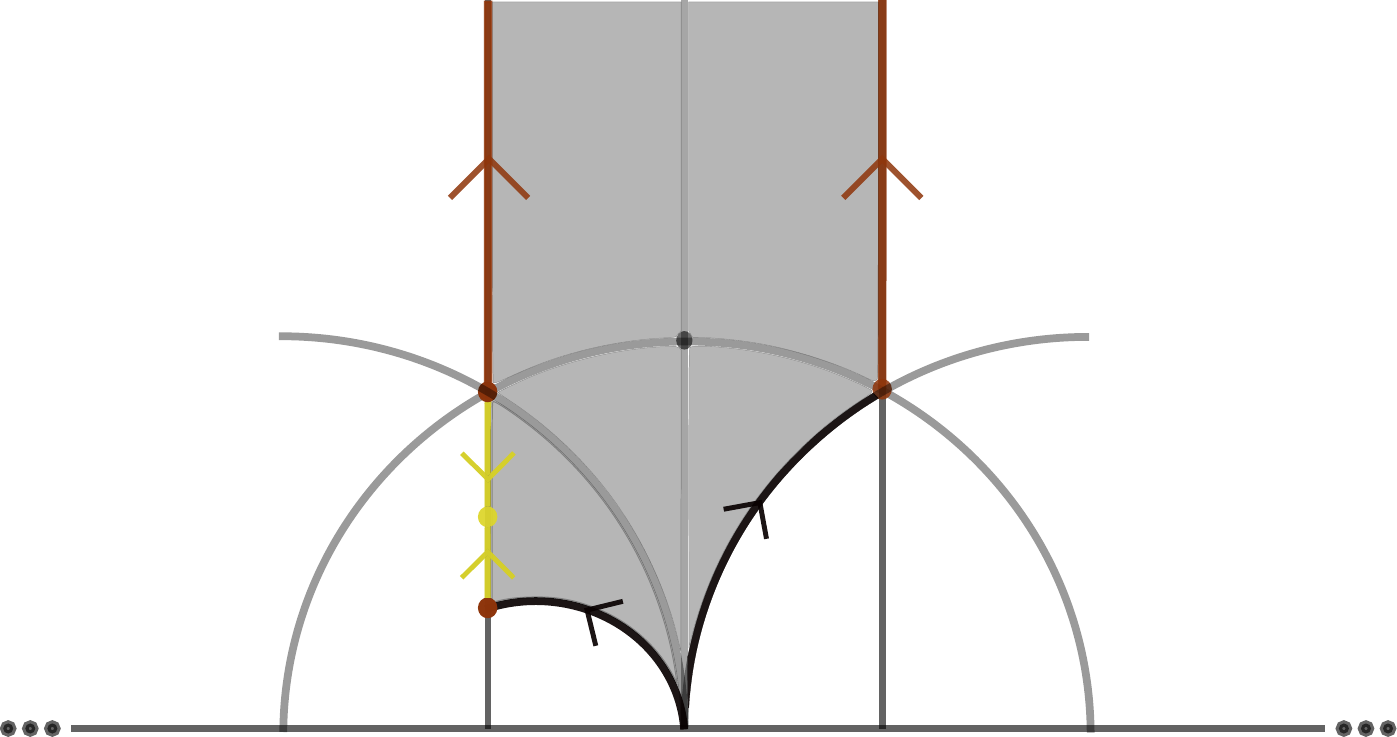}}
\caption{The fundamental domain of $\Gamma_0(2)^+$ on $\mathbb{H}^2$.}
\label{fund0}
\end{center}
\end{figure} 
 

\noindent The borders of the fundamental domain are identified by $T$, $\left( \begin{array}{cc}
1 & 0 \\ -2 & 1 \end{array} \right) $ , and \\
$\left( \begin{array}{cc}
-1 & -1 \\ 2 & 1 \end{array} \right) \in  \Gamma_0(2)^+$. Doing the border identifications we have an orbifold which is homeomorphic to a cylinder. Therefore the moduli space of flat metrics for $B_1$ is 
\begin{equation*}
\begin{aligned}
\mathcal{M}_{flat}(B_1) & \cong (\R^+ \times \mathbb{H}^2 / \Gamma_0(2)^+ ) \times \R^+, \\
             & \cong \mathbb{S}^1\times \R \times (\R^+)^2. 
\end{aligned}
\end{equation*}

\noindent For $B_2$, we have
\begin{equation*}
\begin{aligned}
\mathcal{M}_{flat}(B_2) & =  ( \ort(2) \backslash \G(2, \R) / \Gamma(2)\cdot \left\langle Y \right\rangle) \times \R^{+} \\
        & \cong (\R^+ \times \mathbb{H}^2 / \Gamma(2)\cdot \left\langle Y \right\rangle^+) \times \R^+, \\
\end{aligned}
\end{equation*}
where $\Gamma(2)\cdot \left\langle Y \right\rangle^+$ are the matrices in $\Gamma(2)\cdot \left\langle Y \right\rangle$ with positive determinant. We compute the fundamental domain of $\Gamma(2)\cdot \left\langle Y \right\rangle^+$ on $\mathbb{H}^2$:

1. With a similar procedure as in the case of $\Gamma_0(2)$, we obtain that \\
$[ \text{SL}(2, \Z) : \Gamma(2)\cdot \left\langle Y \right\rangle^+] =3.$  

2. The representatives we choose in SL$(2, \Z)$ for $\Gamma(2)\cdot \left\langle Y \right\rangle^+$ are

$\gamma_1= \Id$, \hspace{0.5cm} $\gamma_2=  \left( \begin{array}{cc}
1 & 1 \\ 0 & 1 \end{array} \right)= T$, \hspace{0.5cm} $\gamma_3= \left( \begin{array}{cc}
1 & -1 \\ 1 & 0 \end{array} \right)=TS$.

3. Since we already expressed the representatives in terms of the generators $T$ and $S$, we can apply them easier to the fundamental domain of SL$(2, \Z)$. In this way we obtain the fundamental domain for $\Gamma(2)\cdot \left\langle Y \right\rangle^+$ on $\mathbb{H}^2$, as shown in Figure \ref{fundE}.


\hspace{3cm}
\begin{figure}[h]
\begin{center}
\scalebox{0.3}[0.3]{\includegraphics{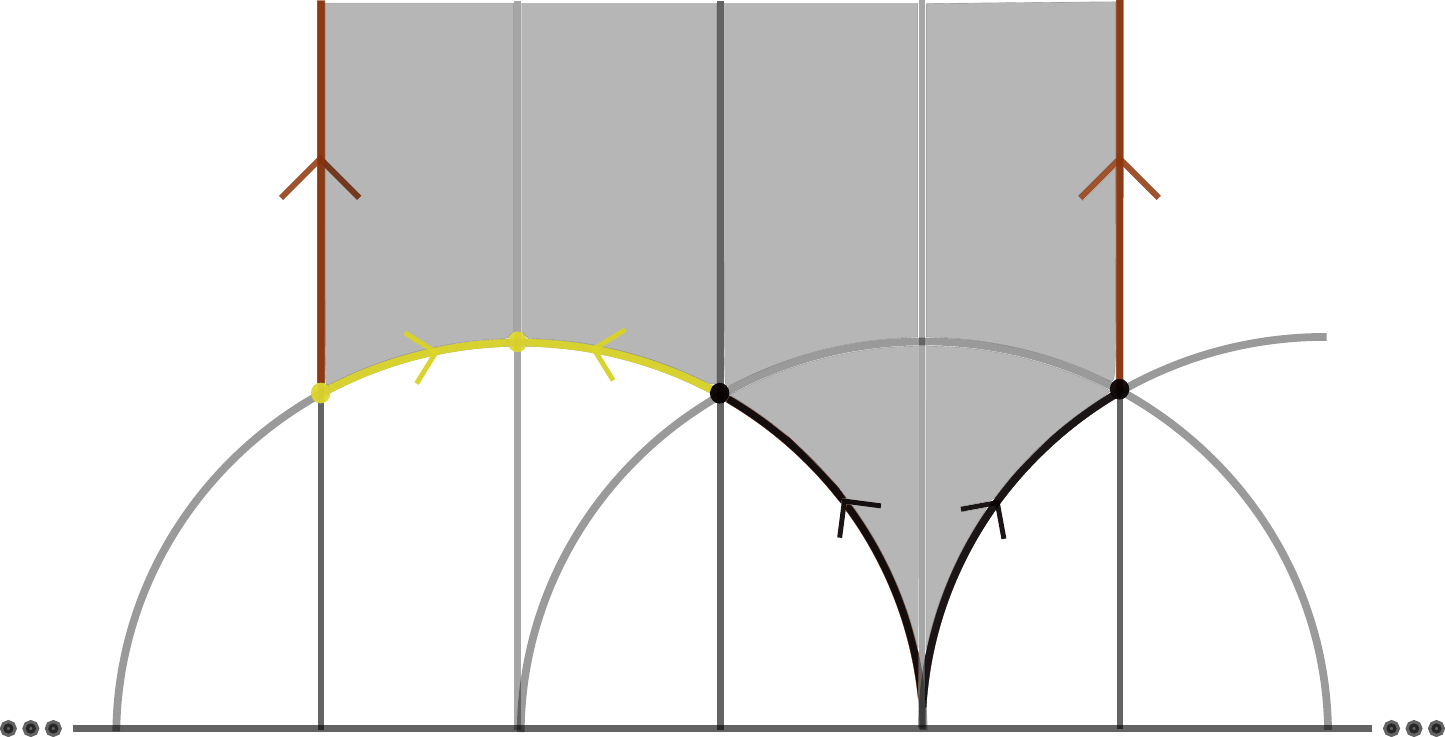}}
\caption{The fundamental domain of $\Gamma(2)Y^+$ on $\mathbb{H}^2$.}
\label{fundE}
\end{center}
\end{figure}

\noindent We notice that the fundamental domain of $\Gamma(2)\cdot \left\langle Y \right\rangle^+$ is quite similar to the one of $\Gamma_0(2)^+$ and it is also homeomorphic to a cylinder. Therefore the moduli space of flat metrics for $B_2$ is

\begin{equation*}
\begin{aligned}
\mathcal{M}_{flat}(B_2) & \cong (\R^+ \times \mathbb{H}^2 / \Gamma(2)\cdot \left\langle Y \right\rangle^+) \times \R^+ \\
        & \cong \mathbb{S}^1\times \R \times (\R^+)^2. 
\end{aligned}
\end{equation*}
   
\end{proof}

\begin{proof}[Proof of Corollary \ref{top4}]
The proof is done case by case. We use some of the descriptions of Theorem \ref{mod4}.

We explain why the following moduli spaces of flat metrics are non-contractible:

\noindent The case of $O^4_2$, since 
\begin{equation*}
\begin{split}
\mathcal{M}_{flat} & = (\text{O}(2)\backslash \text{GL}(2,\mathbb{R}) /\text{GL}(2,\mathbb{Z})) \times (\text{O}(2)\backslash \text{GL}(2,\mathbb{R}) /\Gamma_0(2)) \\ 
        & \cong (\R^+ \times \mathbb{H}^2/\text{SL}(2, \Z)) \times (\R^+ \times \mathbb{H}^2 /\Gamma_0(2)^+)  \\ 
        &  \cong (\R^+)^2 \times \mathbb{S}^2 \setminus \{*\} \times \mathbb{S}^1\times \R . 
\end{split}
\end{equation*}
This double cosets are studied in Theorem \ref{top3}.

\noindent The case of $O^4_7$, since
\begin{equation*}
\begin{split}
 \mathcal{M}_{flat} & = \left( \ort(2) \backslash \G(2, \R) / \Gamma(2) \right) \times \R^+   \\
             & \cong (\R^+ \times \mathbb{H}^2/ \Gamma(2)^+) \times \R^+ 
\end{split}
\end{equation*} 
where $\Gamma(2)^+$ are the matrices in $\Gamma(2)$ with positive determinant. As before, we compute the fundamental domain of $\Gamma(2)^+$ on $\mathbb{H}^2$:

1. $[ \text{SL}(2, \Z) : \Gamma(2)^+] =6 $ (\cite{shimura}, pages 20-22).

2. The representatives that we choose in SL$(2, \Z)$ for $\Gamma(2)^+$ are\\
\noindent $\gamma_1= \Id$, $\gamma_2=  \left( \begin{array}{cc}
0 & -1 \\ 1 & 0 \end{array} \right)= S$, $\gamma_3= \left( \begin{array}{cc}
0 & -1 \\ 1 & 1 \end{array} \right)=ST$, $\gamma_4= \left( \begin{array}{cc}
1 & 1 \\ 0 & 1 \end{array} \right)= T$, \\
 
\noindent  $\gamma_5=  \left( \begin{array}{cc}
1 & -1 \\ 1 & 0 \end{array} \right)= TS$,  
$\gamma_6= \left( \begin{array}{cc}
1 & -2 \\ 1 & -1 \end{array} \right)=TST^{-1}$.

3. Since we already expressed the representatives in terms of the generators $T$ and $S$, we can apply them easier to the fundamental domain of SL$(2, \Z)$. In this way we obtain the fundamental domain for $\Gamma(2)^+$ on $\mathbb{H}^2$, as shown in Figure \ref{fund2}.

\begin{figure}[h]
\begin{center}
\scalebox{0.3}[0.3]{\includegraphics{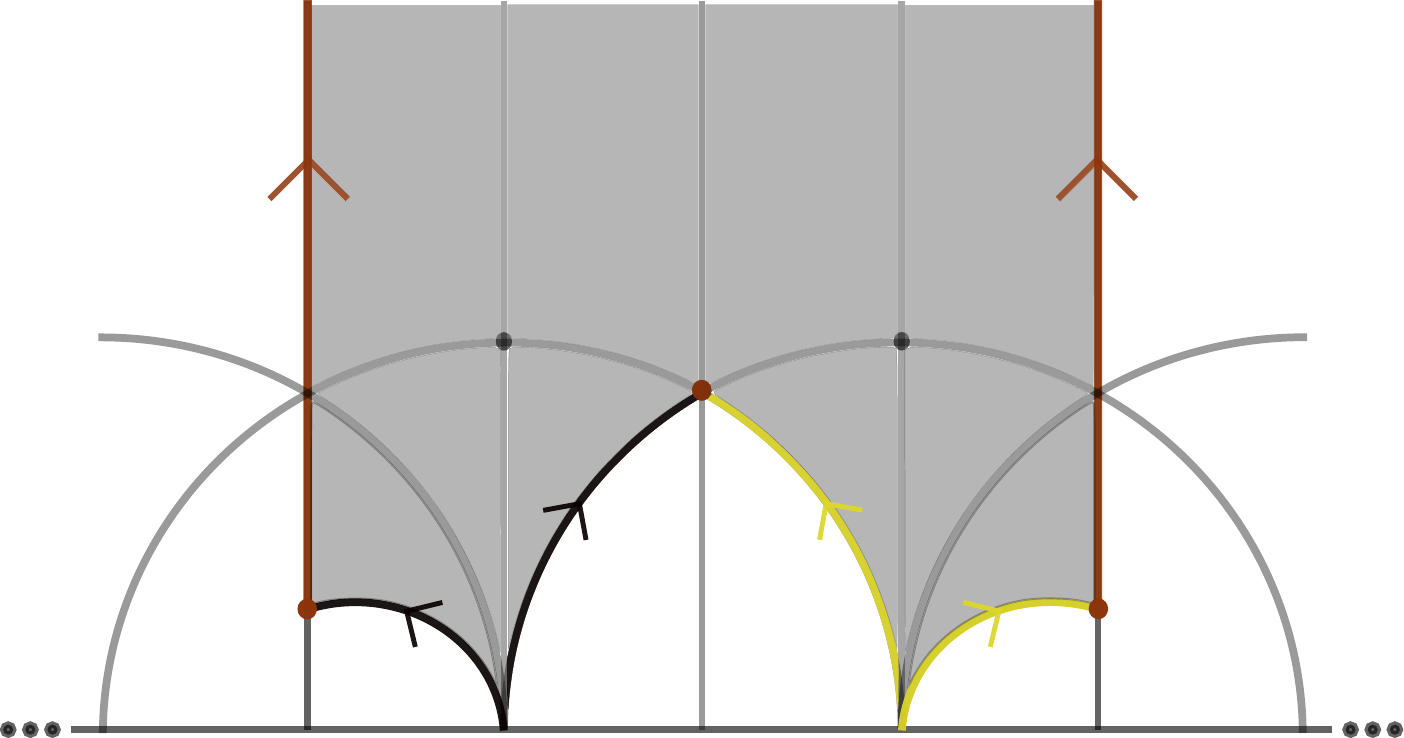}}
\caption{The fundamental domain of $\Gamma(2)^+$ on $\mathbb{H}^2$.}
\label{fund2}
\end{center}
\end{figure} 


\noindent The borders of the fundamental domain are identified by $T^2$, $\left( \begin{array}{cc}
1 & 0 \\ -2 & 1 \end{array} \right) $, and \\
$\left( \begin{array}{cc}
-3 & 2 \\ -2 & 1 \end{array} \right) \in \Gamma(2)^+$. Making the border identifications, as shown in Figure \ref{fund2homeo}, we have an orbifold which is homeomorphic to a 3-punctured sphere.

\begin{figure}[h]
\begin{center}
\scalebox{0.3}[0.3]{\includegraphics{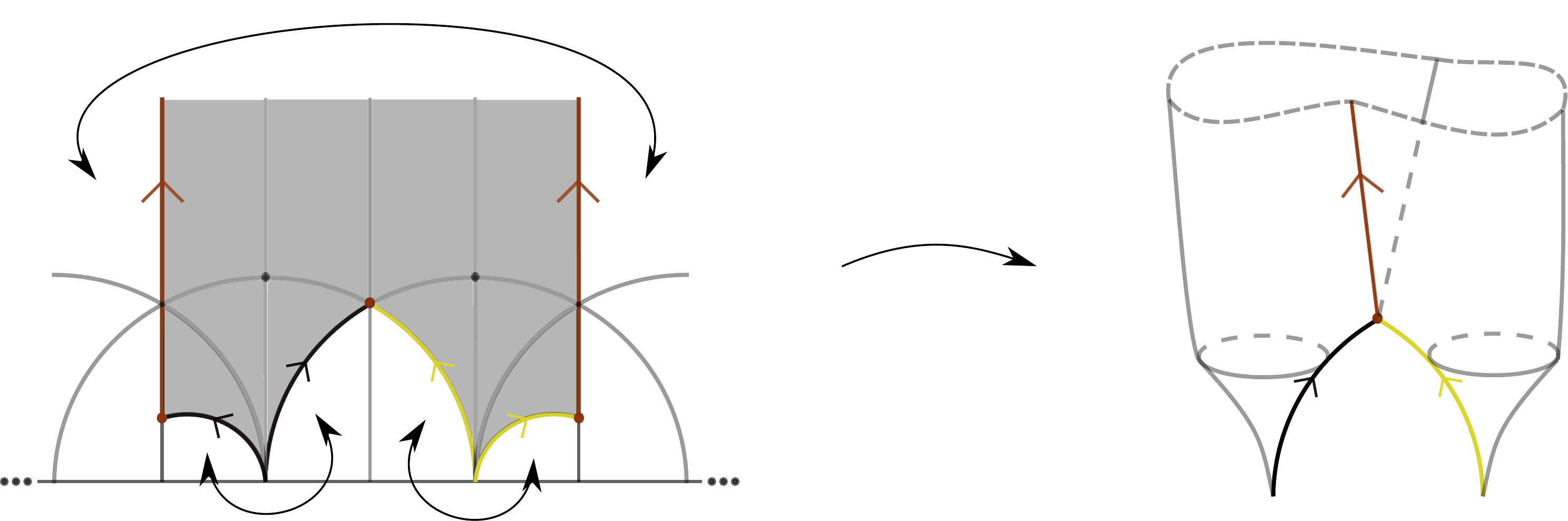}}
\caption{The border identifications.}
\label{fund2homeo}
\end{center}
\end{figure} 

Therefore we can conclude that the moduli space of flat metrics for $O^4_7$ is 
\begin{equation*}
\begin{split}
 \mathcal{M}_{flat} & \cong (\R^+ \times \mathbb{H}^2/ \Gamma(2)^+) \times \R^+ \\
             &  \cong \text{3-punctured sphere} \times (\R^+)^2. 
\end{split}
\end{equation*} 

\noindent The case of $N^4_{14}$ is contractible, because we have \\
$\mathcal{M}_{flat}(N^4_{14}) = \text{O}(3)\backslash \text{GL}(3,\mathbb{R}) /\G(3, \Z) \times \R^+ $, and that double quotient is contractible; see Soulé \cite{soule}. 

\noindent The cases of $N^4_{15}$, $N^4_{16}$, $N^4_{19}$, $N^4_{20}$, and $N^4_{21}$ are also contractible, because $\mathcal{M}_{flat} = (\R^+)^3 .$

\end{proof}

\bibliographystyle{siam}


\end{document}